\documentclass[final, 3p, times]{elsarticle}
\usepackage[utf8]{inputenc}

\usepackage{framed,multirow}
\usepackage{geometry}
\usepackage{amssymb}
\usepackage{amsthm}
\usepackage{amsmath}
\usepackage{amsfonts,amssymb}
\usepackage{enumerate}
\usepackage{booktabs}
\usepackage{graphicx}
\usepackage{multirow}
\usepackage{mathtools}
\usepackage{color}
\usepackage{algorithm}
\usepackage{algpseudocode}
\usepackage{appendix}
\usepackage{stmaryrd}
\usepackage{subfigure}
\usepackage{mathrsfs}
\usepackage{threeparttable}
\usepackage{diagbox}
\usepackage{hyperref}
\usepackage{makecell}

\newcommand{\norm}[1]{\left\lVert#1\right\rVert}
\renewcommand{\vec}[1]{\boldsymbol{#1}}
\newcommand{\y}{\mathbf{y}}
\newcommand{\ind}{i}
\newcommand{\Ind}{I}

\usepackage{ulem}

\newtheorem{theorem}{Theorem}
\newtheorem{remark}{Remark}

\begin{document}

\begin{frontmatter}
\title{Solving parametric elliptic interface problems via interfaced operator network}
\author[1,2]{Sidi Wu}
\author[2]{Aiqing Zhu}
\author[2]{Yifa Tang}
\author[2]{Benzhuo Lu\corref{cor}}
\address[1]{School of Mathematical Sciences, Peking University, Beijing 100871, China}
\address[2]{LSEC, ICMSEC, Academy of Mathematics and Systems Science, Chinese Academy of Sciences, Beijing 100190, China}
\cortext[cor]{Corresponding author. Email addresses: bzlu@lsec.cc.ac.cn}
\date{}

\begin{abstract}
Learning operators mapping between infinite-dimensional Banach spaces via neural networks has attracted a considerable amount of attention in recent years. In this paper, we propose an interfaced operator network (IONet) to solve parametric elliptic interface PDEs, where different coefficients, source terms, and boundary conditions are considered as input features. To capture the discontinuities in both the input functions and the output solutions across the interface, IONet divides the entire domain into several separate subdomains according to the interface and uses multiple branch nets and trunk nets. Each branch net  extracts latent representations of input functions at a fixed number of sensors on a specific subdomain, and each trunk net  is responsible for output solutions on one subdomain. Additionally, tailored physics-informed loss of IONet is proposed to ensure physical consistency, which greatly reduces the training dataset requirement and makes IONet effective without any paired input-output observations inside the computational domain. Extensive numerical studies demonstrate that IONet outperforms existing state-of-the-art deep operator networks in terms of accuracy and versatility.

\end{abstract}
\end{frontmatter}
\textbf{Keywords}: Parametric elliptic interface problems; Interfaced operator network; Operator regression; Mesh-free method.

\section{Introduction}
Elliptic interface problems have widespread applications across various fields, including fluid mechanics \cite{sussman1999efficient,fadlun2000combined}, materials science \cite{liu2020moment,wang2019petrov}, electromagnetics \cite{hesthaven2003high}, biomimetics \cite{lu2008recent,ji2018finite}, and flow in porous media \cite{philip1970flow}. Accurate modeling and rapid evaluation of these differential equations are critical in both scientific research and engineering applications. Many computational tasks arising in science and engineering often involve repeated evaluation of the outputs of an expensive forward model for many statistically similar inputs. These tasks, known as parametric PDE problems, encompass various areas such as inverse problems, control and optimization, risk assessment, and uncertainty quantification \cite{khoo2021solving,zhu2019physics}. When dealing with parametric PDEs with discontinuous coefficients across certain interfaces, i.e., parametric interface problems, the low global regularity of the solution and the irregular geometry of the interface give rise to additional challenges, particularly for problems with non-smooth interfaces containing geometric singularities such as sharp edges, tips, and cusps.

Consider an open and bounded domain $\Omega \subset\mathbb{R}^d$ with a Lipschitz boundary $\partial\Omega$. The domain $\Omega$ is separated into two disjoint subdomains, $\Omega_1$ and $\Omega_2$, by an interface $\Gamma$. A sketch of the computational domain considered in 2D is shown in Fig. \ref{fig: computational domain}. Then parametric second-order linear elliptic interface problems are of the form:
\begin{subequations}  \label{eq: interface pde}
\begin{align}
     -\nabla\cdot(a\nabla u)+b u &=  f,\quad \text{in} \ \Omega\setminus\Gamma, \label{eq:1A}\\
    \llbracket u\rrbracket &=g_D,\quad  \text{on} \ \Gamma \label{eq:1B},\\
    \llbracket a\nabla u\cdot \mathbf{n} \rrbracket &=g_N, \quad  \text{on} \ \Gamma \label{eq:1C}, \\
    u&=h,\quad \text{on} \ \partial\Omega, \label{eq:1D}
\end{align}
\end{subequations}
where $\mathbf{n}$ denotes the outward unit normal vectors of the interface $\Gamma$ (from $\Omega_1$ to $\Omega_2$), $\llbracket \cdot \rrbracket$ denotes the jump across the interface, for a point $\mathbf{x}^\gamma\in \Gamma$,
\begin{equation*}
\begin{aligned}
\llbracket u\rrbracket (\mathbf{x}^\gamma):&= \lim_{\mathbf{x}\in \Omega_2 \atop \mathbf{x} \to \mathbf{x}^\gamma } u(\mathbf{x}) - \lim_{\mathbf{x}\in \Omega_1 \atop \mathbf{x} \to \mathbf{x}^\gamma } u(\mathbf{x}),\\
\llbracket a\nabla u\cdot \mathbf{n} \rrbracket (\mathbf{x}^\gamma):&=  \lim_{\mathbf{x}\in \Omega_2 \atop \mathbf{x} \to \mathbf{x}^\gamma } a(\mathbf{x})\nabla  u(\mathbf{x}) \cdot \mathbf{n} - \lim_{\mathbf{x}\in \Omega_1 \atop \mathbf{x} \to \mathbf{x}^\gamma } a(\mathbf{x})\nabla  u(\mathbf{x}) \cdot \mathbf{n}.
\end{aligned}
\end{equation*}
Here,  { $g_D(\mathbf{x}): \Gamma\rightarrow \mathbb{R}$ and $g_N(\mathbf{x}): \Gamma\rightarrow \mathbb{R}$ are the interface conditions, and $h(\mathbf{x}): \partial\Omega\rightarrow \mathbb{R}$ is the boundary condition;} the coefficient $a(\mathbf{x}):\Omega\rightarrow \mathbb{R}$ is continuous and positive in each of the subdomains but discontinuous across the interface; the coefficient $b(\mathbf{x}):\Omega\rightarrow \mathbb{R}$ and source $f(\mathbf{x}):\Omega\rightarrow \mathbb{R}$ are continuous in each of the subdomains but may be discontinuous across the interface. Additionally, we will also consider a nonlinear example, i.e., replacing $bu$ here with $b(u)$.
The latent solution $u(\mathbf{x}): \Omega\rightarrow \mathbb{R}$ to this problem typically has higher regularity in each subdomain, but lower global regularity across the whole domain, even with discontinuities at the interface.
Solving these parametric elliptic interface problems requires
learning the solution operator that maps variable PDE parameters such as the coefficient $a(\mathbf{x})$ and the source term $f(\mathbf{x})$ directly to the corresponding solution $u$. This paper introduces a novel operator network for approximating operators involving discontinuities in both input and output functions. We then demonstrate its effectiveness in approximating the solution operator of parametric elliptic interface problems.

\begin{figure}[htbp]
	\centering
	\scalebox{0.8}{\includegraphics{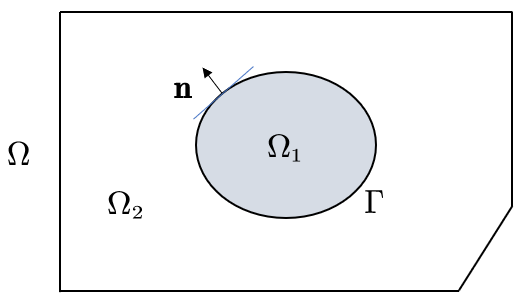}}
	\caption{Domain $\Omega$, its subdomains $\Omega_1$, $\Omega_2$. The interface $\Gamma$ divides $\Omega$ into two disjoint subdomains.}
\label{fig: computational domain}
\end{figure}

Classical numerical methods for solving elliptic interface problems can be roughly divided into two categories: interface-fitted methods and interface-unfitted methods. The first type of approach is suitable for solving PDE problems defined in complex domains.  The methods in this category include classical finite element method (FEM) \cite{babuvska1970finite,bramble1996finite,chen2020bilinear}, boundary element method (BEM) \cite{zhang2019dashmm}, weak Galerkin method \cite{mu2016new}, and so on. To maintain optimum convergence behavior, these methods require the mesh surface to be aligned with the interface. This alignment ensures that interface conditions are correctly applied, enhancing the accuracy of numerical solutions. However, generating interface-fitted meshes for irregular domains or interfaces could result in significant computational costs \cite{liu2018efficient}, especially when high accuracy is required. To alleviate the burden of mesh generation, many works employ interface-unfitted meshes (e.g., a uniform Cartesian mesh) to discretize the computational domain and enforce interface conditions by modifying finite difference stencils or finite element bases near the interface. For instance, the immersed boundary method (IBM) \cite{peskin2002immersed}, the immersed interface method (IIM) \cite{leveque1994immersed}, the immersed finite element method \cite{chen2009adaptive}, the ghost fluid method (GFM) \cite{fedkiw1999non} and its improvement (xGFM) \cite{egan2020xgfm}, the cartesian grid finite volume approach (FVM) \cite{bochkov2020solving,thacher2023high}, the matched interface and boundary method \cite{xia2011mib}, the extended finite element methods (XFEM) \cite{babuvska2012stable,liu2020interface}, and references therein. In general, the numerical solution of these methods becomes more accurate with mesh refinement, but also more time consuming.

Besides the mesh-based methods, there are also numerous efforts focusing on mesh-free numerical methods for interface problems, such as direct meshless local Petrov-Galerkin method \cite{taleei2014direct}, the global RBF-QR collection method \cite{gholampour2021global}, the local RBF meshless methods \cite{ahmad2020local}, and the meshless method based on pascal polynomials and multiple-scale approach \cite{orucc2021efficient}.
Alternatively, there is a growing interest in utilizing neural network-based methods to solve elliptic interface problems. For instance, \cite{wang2020mesh} employed a shallow neural network to remove the inhomogeneous boundary conditions and developed a deep Ritz-type approach to solve the interface problem with continuous solutions.
An important development in this direction is the combination of deep learning and domain decomposition methods due to the observation that the solution to the interface problem is typically piece-wise continuous. The solution to Eq. \eqref{eq: interface pde} can be approximated by minimizing a loss function derived from either the least squares principle \cite{he2022mesh,wu2022interfaced} or the variational principle \cite{guo2022deep,qi2023dirichlet}. Moreover, adaptively setting appropriate penalty weights among different terms in the loss function could improve accuracy \cite{wu2022interfaced,berg2018aunified,jagtap2020extended}. And specially designed neural network structures, such as incorporating multi-scale features \cite{liu2020MscaleDNN,ying2023multi} and augmenting extra feature input \cite{hu2022discontinuity,lai2022shallow,tseng2022cusp}, are also able to further enhance the performance of neural models.

Although these numerical methods have been shown to be effective to some extent, they are only employed to solve a given instance of the elliptic interface problem \eqref{eq: interface pde}, where the coefficient functions $a(\mathbf{x})$ and $b(\mathbf{x})$, the forcing term $f(\mathbf{x})$,  the interface conditions $g_D(\mathbf{x})$ and $g_N(\mathbf{x})$, and the boundary condition $h(\mathbf{x})$ are given in advance.
In other words, these methods treat a PDE with different parameters as different tasks, each of which needs to be solved end-to-end, which is computationally expensive and time-consuming. To address these challenges, one approach is to employ a reduced-order model that leverages a set of high-fidelity solution snapshots to construct rapid emulators \cite{lucia2004reduced,quarteroni2015reduced}. However, the validity of this method relies on the assumption that the solution set is contained approximately in a low-dimensional manifold, which could potentially lead to compromised accuracy and generalization performance \cite{majda2018strategies,ye2023meta}.

Recently, as an emerging paradigm in scientific machine learning, several operator neural networks, such as PDE-Nets \cite{long2019pde}, deep operator network (DeepONet) \cite{lu2021learning} and Fourier neural operator (FNO) \cite{li2021fourier}, have been developed to directly learn the solution mapping between two infinite-dimensional function spaces. Although prediction accuracy may be limited, the ability of neural networks to learn from data makes them particularly well suited to this task. Such methods have great potential for developing fast forward and inverse solvers for PDE problems, and have shown good performance in building surrogate models for many types of PDEs, including the Burgers' equation \cite{lu2021learning,jin2022mionet}, Navier-Stokes equations \cite{li2021fourier,lu2022comprehensive}, Darcy flow \cite{lu2022comprehensive}, diffusion-reaction PDE \cite{wang2021learning}, and so on.

Despite the aforementioned success, these operator learning methods typically have difficulty effectively capturing the discontinuities of input and output functions due to the following two reasons. Theoretically, their approximation theory usually assumes that the input and output functions are continuous \cite{lu2021learning}. Practically, to handle input functions numerically, we typically need to discretize the input functions and evaluate them at a set of locations. This approach may overlook the discontinuity of the true input functions, as the input functions were expected to be continuous. In addition, the output functions are represented by a network which is a continuous approximator. However, for interface problems, the global regularity of the coefficients and the solutions is usually very low, even discontinuous \cite{littman1963regular}. These limit the ability of operator networks to accurately represent and learn the complex behavior associated with interface problems.

To address these limitations, in this paper we propose a novel mesh-free method for approximating the solution operator of parametric elliptic interface problems. Different from existing operator networks, we divide the whole domain into several separate subdomains according to the interface, and leverage multiple branch nets and trunk nets. Specifically, each branch net encodes the input function at a fixed number of sensors in each subdomain, and each trunk net is responsible for output solutions in each subdomain. Such an architecture allows the model to accommodate irregularities in input functions and solutions. In addition, tailored physics-informed loss is proposed to ensure physical consistency, which greatly reduces the requirement for training datasets and makes the network effective without any paired input-output observations in the interior of the computational domain. The proposed method circumvents mesh generation and numerical discretization at the interface(s), thus easily handling problems in irregular domains. And the model can be trained only once for fast simulation with different input functions. Herein, we name this neural operator as the Interfaced Operator Network (IONet). Numerical results show that IONet exhibits better accuracy, as well as generalization properties for various input parameters, compared with state-of-the-art neural models.

The rest of this paper is organized as follows: In Section \ref{Preliminaries}, we review the basic idea of the operator network and the DeepONet method. In Section \ref{main methodology}, we introduce the proposed interfaced neural network in detail. Then, in Section \ref{numerical examples}, we investigate the performance of the proposed methods in several typical numerical examples. Finally, we conclude the paper and discuss some future directions in Section \ref{conclusions}.

\section{Learning operators with neural networks\label{Preliminaries}}

In this section, we briefly introduce the DeepONet model architecture \cite{lu2021learning} and its two extensions, the Multi-input operator network (MIONet) \cite{jin2022mionet} and the Physics-informed DeepONet (PI-DeepONet) \cite{wang2021learning}, for learning nonlinear operators between infinite function spaces.

Let $\mathcal{V}$ and $\mathcal{U}$ be two Banach spaces, and let $\mathcal{G}$ be an operator that maps between these two infinite-dimensional function spaces, i.e., $\mathcal{G}: \mathcal{V}\rightarrow\mathcal{U}$. We assume that for each $v(\mathbf{y}): \mathbf{y}\rightarrow \mathbb{R}$ in $\mathcal{V}$, there exists a unique corresponding output function $u$ in $\mathcal{U}$ that can be represented as $\mathcal{G}(v)(\mathbf{x}): \mathbf{x}\rightarrow \mathbb{R}$. Analogously, in the context of parametric PDE problems, $\mathcal{V}$ and $\mathcal{U}$ are denoted as the input function space and the solution space, respectively.
Following the original works of \cite{lu2021learning,wang2021learning},
an unstacked DeepONet $\mathcal{G}_{\mathbf{\theta}}$ is trained to approximate the target solution operator $\mathcal{G}$, where $\mathcal{G}_{\mathbf{\theta}}$ prediction of a function (an input parameter) $v\in\mathcal{V}$ evaluated at a point $\mathbf{x}$ can be expressed as
\begin{equation*}
\begin{aligned}
\mathcal{G}_{\mathbf{\theta}}(v)(\mathbf{x}) &=
\underbrace{\mathcal{N}_{b}(v(\mathbf{y}_1), v(\mathbf{y}_2), \cdots ,v(\mathbf{y}_m))^T}_{branch\ net}
\underbrace{\mathcal{N}_{t}(\mathbf{x})}_{trunk\ net}
+\underbrace{b_0}_{bias}\\
&=\sum_{k=1}^K b_k t_k+b_0,
\end{aligned}
\end{equation*}
where $\mathbf{\theta}$ denotes all the trainable parameters, i.e., the set consisting of the parameters in the branch net  $\mathcal{N}_{b}$ and the trunk net  $\mathcal{N}_{t}$ and the bias $b_0\in\mathbb{R}$. Here, $[b_1, b_2, \cdots , b_K]^T\in\mathbb{R}^K$ denotes the output of $\mathcal{N}_{b}$ as a feature embedding of the input function $v$, $[t_1, t_2, \cdots , t_K]^T\in\mathbb{R}^K$ represents the output of $\mathcal{N}_{t}$, and $\{\mathbf{y}_1, \mathbf{y}_2, \cdots , \mathbf{y}_m\}$ is a collection of fixed point locations referred to as ``sensors", where we discretize the input function $v$.
DeepONets are capable of approximating arbitrary continuous operators \cite{chen1995universal, lu2021learning}, making them a powerful tool in the field of scientific computing.

MIONet \cite{jin2022mionet} extends the architecture and approximation theory of DeepONet to the case of operators defined on multiple Banach spaces. Let $\mathcal{G}$ be a multi-input operator defined on the product of Banach spaces:
\begin{equation*}
\mathcal{G}:\mathcal{V}_1\times\mathcal{V}_2\times\cdots \times\mathcal{V}_q\rightarrow\mathcal{U},
\end{equation*}
where $\mathcal{V}_1, \mathcal{V}_2, \cdots , \mathcal{V}_q$ are $q$ different input Banach spaces that can be defined on different domains, and $\mathcal{U}$ denotes the output Banach space. Then, when we employ a MIONet $\mathcal{G}_\mathbf{\theta}$ to approximate the operator $\mathcal{G}$, for a given input function $(v_1, v_2,\cdots , v_q)\in\mathcal{V}_1\times\mathcal{V}_2\times\cdots \times\mathcal{V}_q$, the prediction of $\mathcal{G}_{\mathbf{\theta}}(v_1, v_2,\cdots , v_q)$
at a point $\mathbf{x}$ is formulated as
\begin{equation*}
\begin{aligned}
\mathcal{G}_{\mathbf{\theta}}(v_1, v_2,\cdots , v_q) (\mathbf{x})&=
\mathcal{S}\left(\underbrace{\mathcal{N}_{b_1}(\vec{v_1})}_{branch_1}\odot\underbrace{\mathcal{N}_{b_2}(\vec{v_2})}_{branch_2}\odot\cdots \odot\underbrace{\mathcal{N}_{b_q}(\vec{v_q})}_{branch_{q}}\odot
\underbrace{\mathcal{N}_{t}(\mathbf{x})}_{trunk}\right)
+\underbrace{b_0}_{bias}\\
&=\sum_{k=1}^K t_k\prod \limits_{i=1}^q b^i_k +b_0.
\end{aligned}
\end{equation*}
Here, $\mathcal{S}$ is the sum of all the components of a vector, and $\odot$ represents the Hadamard product. Each input function $v_i$ is projected onto finite-dimensional spaces $\mathbb{R}^{m_i}$ as $\vec{v_i}:=[v_i(\mathbf{y}^{i}_{1}), v_i(\mathbf{y}^{i}_{2}), \cdots ,v_i(\mathbf{y}^{i}_{m_i})]^T$ in the same manner as in DeepONet, where $\{\mathbf{y}^{i}_{ j}\}_{j=1}^{m_i}$ is the set of sensors in the domain of $v_i$. Similarly,
$[b_1^i, b_2^i,\cdots ,b_K^i]$ and $[t_1, t_2,\cdots ,t_K]$ denote the output of the $i$-th branch net  $\mathcal{N}_{b_i}(\vec{v_i})$ and the trunk net  $\mathcal{N}_{t}(\mathbf{x})$, respectively.

In the framework of vanilla DeepONet, a data-driven (DD) approach is used to train the network. Specifically, the training dataset consists of paired input-output observations, and the trainable parameters $\mathbf{\theta}$ can be identified by minimizing the following empirical loss function:
\begin{equation*}
\text{Loss}(\mathbf{\theta})=\frac{1}{NP}\sum_{n=1}^N\sum_{p=1}^P\left| \mathcal{G}_{\mathbf{\theta}}(v^n)(\mathbf{x}_{n,p})-\mathcal{G}(v^n)(\mathbf{x}_{n,p})\right|^2,
\end{equation*}
where $\{v^n\}_{n=1}^N$ denotes $N$ input functions sampled from the parameter space $\mathcal{V}$. For each input function of DeepONet, the training data points $\{\mathbf{x}_{n,p}\}_{p=1}^P$ are randomly sampled from the computational domain of $\mathcal{G}(v^n)$ and can be set to vary for different $n$.

Given that the DeepONet architecture provides a continuous approximation of the target functions that is independent of the resolution, the derivatives of the output function can be computed during training. This import feature motivated the work of PI-DeepONet \cite{wang2021learning}, where the trainable parameters can be optimized by minimizing the residuals of the governing equations and the corresponding boundary conditions through the use of automatic differentiation \cite{baydin2018automatic}. Specifically, consider a generic parametric PDE expressed as:
\begin{equation*}
\label{general_PDE}
\begin{aligned}
\mathcal{L}(v,u)&=0, \ &&\text{in} \ \Omega,\\
 u&=h, \  &&\text{on} \ \partial\Omega,
\end{aligned}
\end{equation*}
where $v$ and $u$ denote the input function and latent solution, respectively. Then, the physics-informed loss function of PI-DeepONet can be formulated as
\begin{equation*}
    \text{Loss}(\mathbf{\theta})= \lambda_r\text{Loss}_{r}(\mathbf{\theta})+\lambda_b\text{Loss}_{b}(\mathbf{\theta}).
\end{equation*}
Here, $\lambda_r$ and $\lambda_b$ are non-negative weights, the loss term
\begin{equation*}
    \text{Loss}_{r}(\mathbf{\theta}) = \frac{1}{NP_r}\sum_{n=1}^N\sum_{p=1}^{P_r} \left| \mathcal{L}(v^n, \mathcal{G}_{\mathbf{\theta}}(v^n))(\mathbf{x}_{n,p}^r) \right|
\end{equation*}
forces the operator network to satisfy the underlying physical constraints, and
\begin{equation*}
    \text{Loss}_{b}(\mathbf{\theta}) = \frac{1}{NP_b}\sum_{n=1}^N\sum_{p=1}^{P_b}\left| \mathcal{G}_{\mathbf{\theta}}(v^n)(\mathbf{x}_{n,p}^b)-h(\mathbf{x}_{n,p}^b) \right|
\end{equation*}
penalizes the violation of the boundary conditions,
where $\{\mathbf{x}_{n,p}^r\}_{p=1}^{P_r}$ and $\{\mathbf{x}_{n,p}^b\}_{p=1}^{P_b}$
denote the training data points randomly sampled from the interior and the boundary of the domain $\Omega$, respectively.

\section{Interfaced operator network}\label{main methodology}
In this section, we discuss neural network-based methods for numerically solving parametric interface problems. The main idea of our new method is to approach the solution operator through multiple suboperators while remaining consistent with the potential physical constraints. To simplify the explanation, we mainly present our method for the case of two subdomains. Note that this setting can be easily generalized to a multi-domain scenario, depending on the number of distinct domains involved. Specifically, we consider Eq. \eqref{eq: interface pde} as a parametric interface problem of general form. In the following, we illustrate our method with the example of learning the solution operator $\mathcal{G}$ that maps the coefficient function $a(\mathbf{x})$ to the solution $u(\mathbf{x})$ of Eq. \eqref{eq: interface pde}, i.e., $\mathcal{G}: a(\mathbf{x}) \rightarrow u(\mathbf{x})$.

\subsection{Network architecture of IONet}
To preserve the inherent discontinuity of interface problems, we decompose the computational domain into two subdomains according to the interface and leverage two operator networks that share some parameters, each of which is responsible for the solution in one subdomain. In particular, the IONet architecture is given as follows:
\begin{equation}\label{eq:defino}
\mathcal{G}_{\mathbf{\theta}}(a)(\mathbf{x}) = \left\{\begin{aligned}
&\mathcal{G}_{\mathbf{\theta}}^1(a)(\mathbf{x}),\quad \text{if}\ \mathbf{x} \in \Omega_1,\\
&\mathcal{G}_{\mathbf{\theta}}^2(a)(\mathbf{x}),\quad \text{if}\ \mathbf{x} \in \Omega_2,\\
\end{aligned}\right.
\end{equation}
where $a$ is the input function and $\mathbf{x}$ denotes the location where the output function is evaluated. Note that input functions are discretized and evaluated at a set of sensors typically. To retain the irregularity of the input function on the interface, we divide the set of sensors according to the interface and use two branch nets, denoted as $\mathcal{N}_{b_1}$ and $\mathcal{N}_{b_2}$, to extract latent representations of input functions on the corresponding subdomains. Similar to the vanilla DeepONet \cite{lu2021learning}, within each suboperator $\mathcal{G}^i_{\mathbf{\theta}}$, we use a trunk net  denoted as $\mathcal{N}^i_{t}$ to extract continuous input coordinates where the output functions are evaluated. Finally, following the MIONet \cite{jin2022mionet}, we merge the outputs of all the sub-networks through a Hadamard product and a summation, followed by the addition of a bias in the last stage. More specifically, the suboperator in $\mathcal{G}_{\mathbf{\theta}}$ (\ref{eq:defino}) is constructed as follows:
\begin{equation}\label{eq:defino2}
\begin{aligned}
\mathcal{G}^i_{\mathbf{\theta}}(a) (\mathbf{x})&=
\mathcal{S}\left(\underbrace{\mathcal{N}_{b_1}(a(\y_1^1),\cdots, a(\y^1_{m_1}))}_{branch_1}\odot\underbrace{\mathcal{N}_{b_2}(a(\y^2_1),\cdots, a(\y^2_{m_2}))}_{branch_2}\odot
\underbrace{\mathcal{N}^i_{t}(\mathbf{x})}_{trunk}\right)
+\underbrace{b_0^i}_{bias}\\
&=\sum_{k=1}^K{t_k^i b_{1k} b_{2k}} +b_0^i.
\end{aligned}
\end{equation}
Here, $\mathbf{\theta}$ denotes the trainable  parameters in this architecture. For $i=1,2$, $\{\y_j^i\}_{j=1}^{m_i}$ represents the collection of sensors for evaluating $a(\mathbf{x})$ in subdomain $\Omega_i$, $[b_{i1}, b_{i2},\cdots ,b_{iK}]$ and $[t_1^i, t_2^i,\cdots ,t_K^i]$ denote the output features of the branch nets $\mathcal{N}_{b_i}$ and the trunk net  $\mathcal{N}^i_{t}$, respectively. The network architecture of IONet is schematically visualized on the left side of Fig. \ref{fig: ino neural architecture}. To demonstrate the capability and performance alone, we apply the simplest feedforward neural networks (FNNs) as the branch and trunk nets in this paper, and we note that other neural networks such as ResNet and CNN can be chosen as the sub-networks in IONet according to specific problems.
\begin{figure}[htbp]
	\centering
	\scalebox{0.48}{\includegraphics{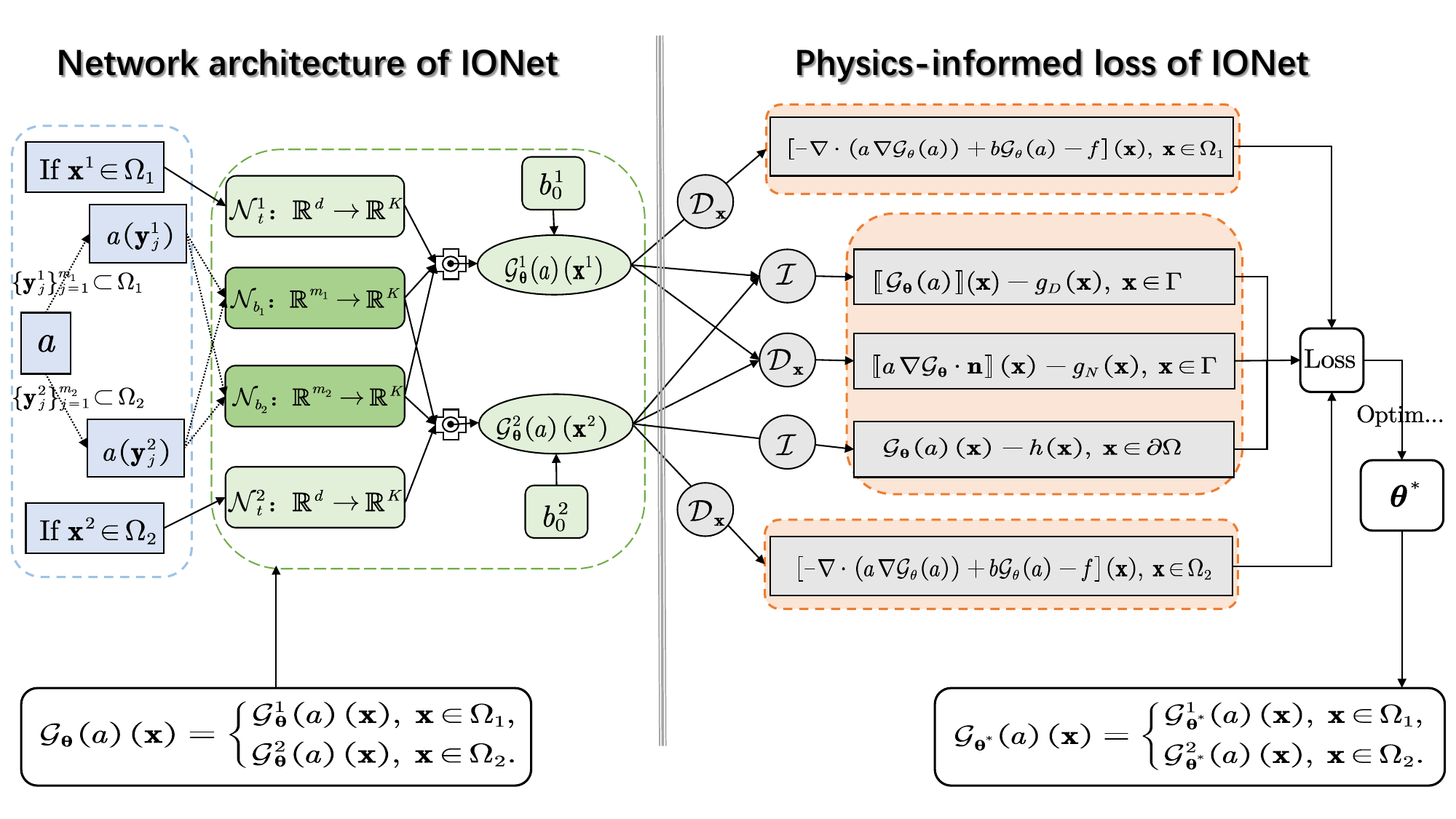}}
	\caption{A schematic diagram of the IONet for solving the parametric elliptic interface problem by minimizing the physics-informed loss function. Here, the input function is the coefficient $a(\mathbf{x})$.}
\label{fig: ino neural architecture}
\end{figure}

The IONet structure can be easily generalized to a multi-domain scenario. Next, we show that IONet is able to approximate arbitrary continuous operators with discontinuous inputs and outputs. For later analysis, we define the following space
\begin{equation*}
X (\Omega) = \bigcap_{\ind=1}^{\Ind} H^{2} (\Omega_{\ind})\bigcap H^{0}(\Omega)
\end{equation*}
equipped with the norm
\begin{equation*}
\norm{u}_{X(\Omega)}= \sum_{\ind=1}^{\Ind}\norm{u}_{H^{2} (\Omega_{\ind})}.
\end{equation*}
Then, the approximation theorem of IONet is given as follows.

\begin{theorem}\label{thm:approximation of ino}
Let $\Omega \subset \mathbb{R}^d$ be a bounded domain, $\Omega_{\ind}$ with $\ind=1,\cdots, \Ind-1$ be disjoint open domains and $\Omega_{\Ind}=\Omega\setminus \bigcup_{\ind=1}^{\Ind-1}\Omega_{\ind}$. Assume $\mathcal{G}: \bigcap_{\ind=1}^{\Ind} C(\Omega_{\ind})\bigcap L^{\infty}(\Omega) \rightarrow X(\Omega)$ is a continuous operator and $T \subset \bigcap_{\ind=1}^{\Ind} C(\Omega_{\ind})\bigcap L^{\infty}(\Omega)$ is a compact set. Then for any $\varepsilon>0$, there exist positive integers $m_{\ind}$, $K$, tanh FNNs $\mathcal{N}_{b_{\ind}}: \mathbb{R}^{m_{\ind}}\rightarrow \mathbb{R}^K$, $\mathcal{N}^{{\ind}}_{t}: \mathbb{R}^{d}\rightarrow \mathbb{R}^K$, and $\y^{\ind}_1,\cdots, \y^{\ind}_{m_{\ind}} \in \Omega_{\ind}$ with $\ind=1,\cdots, \Ind$, such that
\begin{equation*}
\sup_{a\in T} \norm{\mathcal{G}(a)(\cdot) - \mathcal{S} \left( \mathcal{N}_{b_1}(a(\y_1^1),\cdots, a(\y^1_{m_1})) \odot \cdots \odot \mathcal{N}_{b_{\Ind}}(a(\y^{\Ind}_1),\cdots, a(\y^{\Ind}_{m_{\Ind}})) \odot \mathcal{N}^{\ind}_{t}(\cdot) \right) }_{H^2(\Omega_{\ind})} \leq \varepsilon,
\end{equation*}
where $\mathcal{S}$ is the summation of all the components of a vector, and $\odot$ is the Hadamard product.
\end{theorem}

\begin{proof}
The proof can be found in \ref{app:proof of approximation}.
\end{proof}

\begin{remark}
There exist various continuous operator $\mathcal{G}: \bigcap_{\ind=1}^{\Ind} C(\Omega_{\ind})\bigcap L^{\infty}(\Omega) \rightarrow X(\Omega)$. For example, for interface problem \eqref{eq: interface pde}, the operator mapping from the source term $f\in \bigcap_{\ind=1}^{\Ind} C(\Omega_{\ind})\bigcap L^{\infty}(\Omega)$ to the solution $u \in X(\Omega)$ is continuous, due to the estimate \cite{chen1998finite} that $\norm{u}_X \leq C\norm{f}_{L^2(\Omega)} \leq C\sum_{i=1}^I\norm{f}_{C(\Omega_i)}$.
\end{remark}

\subsection{Loss function of IONet}

Similar to DeepONet, a data-driven approach can be used to train IONet and optimize the parameters $\mathbf{\theta}$ by minimizing the following mean square error loss:
\begin{equation}
\begin{aligned}
\label{eq: ddloss of ino}
L_{operator}(\mathbf{\theta})&=\frac{1}{N_o P_o} \sum_{n=1}^{N_o}\sum_{p=1}^{P_o}\left| \mathcal{G}_{\mathbf{\theta}} (a^n_o)(\mathbf{x}_{n,p}^o) -  \mathcal{G}(a^n_o)(\mathbf{x}_{n,p}^o) \right|^2,\\
\end{aligned}
\end{equation}
where $\{a^n_o\}_{n=1}^{N_o}$ denotes $N_o$ input functions sampled from the parameter space; for $n=1, \cdots, N_o$, the training data points $\{\mathbf{x}_{n,p}^o\}_{p=1}^{P_o} \subset \Omega$ denotes the set of locations to evaluate the output function and can be set to vary for different $n$; $\mathcal{G}(a^n_o)(\mathbf{x}_{n,p}^o)$ and $\mathcal{G}_{\mathbf{\theta}}(a^n_o) (\mathbf{x}_{n,p}^o)$ are evaluated values of the output functions of the solution operator $\mathcal{G}$ and IONet $\mathcal{G}_{\mathbf{\theta}}$ at location $\mathbf{x}_{n,p}^o$ when $a^n_o$ is the input function. This type of training method relies on the assumption that there is sufficient labeled data
\begin{equation*}
\left\{\left(a^n_o(\y_1^1),\cdots, a^n_o(\y^1_{m_1}),\; a^n_o(\y^2_1),\cdots, a^n_o(\y^2_{m_2}),\; \mathbf{x}_{n,p}^o,\; \mathcal{G}(a^n_o)(\mathbf{x}_{n,p}^o)\right) \right\}_{n=1, \cdots, N_o,\ p=1, \cdots, P_o}
\end{equation*}
to train the model. However, the costs associated with experimental data acquisition and high-quality numerical simulation are generally expensive. In many practical scenarios, we are inevitably faced with limited or even intractable training data. In the following, we introduce the physics-informed loss function for IONet.

By the definitions of IONet (\ref{eq:defino}) and (\ref{eq:defino2}), the output function of $\mathcal{G}_{\mathbf{\theta}}$, such as
\begin{equation*}
\mathcal{G}_{\mathbf{\theta}}(a^n) (\mathbf{x})= \left\{
\begin{aligned}
&\mathcal{S}\left(\mathcal{N}_{b_1}(a^n(\y_1^1),\cdots, a^n(\y^1_{m_1}))\odot\mathcal{N}_{b_2}(a^n(\y^2_1),\cdots, a^n(\y^2_{m_2}))\odot
\mathcal{N}^1_{t}(\mathbf{x})\right)
+b_0^1,\ \text{if }\mathbf{x} \in \Omega_1,\\
&\mathcal{S}\left(\mathcal{N}_{b_1}(a^n(\y_1^1),\cdots, a^n(\y^1_{m_1}))\odot\mathcal{N}_{b_2}(a^n(\y^2_1),\cdots, a^n(\y^2_{m_2}))\odot
\mathcal{N}^2_{t}(\mathbf{x})\right)
+b_0^2,\ \text{if }\mathbf{x} \in \Omega_2,\\
\end{aligned}\right.
\end{equation*}
has a continuous representation in each subdomain. Provided that the trunk net $\mathcal{N}^i_{t}$ are smooth enough, the derivatives of $\mathcal{G}_{\mathbf{\theta}}(a^n)$ at $\mathbf{x}_{n,p}^i$ can be easily obtained by automatic differentiation \cite{baydin2018automatic}. Inspired by PINN \cite{raissi2019physics} and PI-DeepONet, for given parametric elliptic interface problems, we define
\begin{equation}
    L_{r_i}(\mathbf{\theta}) := \sum_{n=1}^N\sum_{p=1}^{P_i} \bigg|- \nabla \cdot \left(a^n(\mathbf{x}_{n,p}^i) \nabla \mathcal{G}_{\mathbf{\theta}}(a^n)(\mathbf{x}_{n,p}^i)\right) +b(\mathbf{x}_{n,p}^i)\mathcal{G}_{\mathbf{\theta}}(a^n)(\mathbf{x}_{n,p}^i)-f(\mathbf{x}_{n,p}^i)\bigg|^2
    \label{eq: loss term Li}
\end{equation}
and
\begin{equation*}
L_b(\mathbf{\theta}):=\sum_{n=1}^N\sum_{p=1}^{P_b} \left|\mathcal{G}_{\mathbf{\theta}}(a^n)(\mathbf{x}_{n,p}^b)-h(\mathbf{x}_{n,p}^b)\right|^2,
\end{equation*}
where $\{\mathbf{x}_{n,p}^i\}_{p=1}^{P_i}$ with $i=1,2$ and $\{\mathbf{x}_{n,p}^b\}_{p=1}^{P_b}$
are randomly sampled from the subdomain $\Omega_i$ and its boundary $\partial \Omega$, respectively.
Let $L_{\Gamma}(\mathbf{\theta}) = L_{\Gamma_D}+L_{\Gamma_N}$, where
\begin{equation*}
L_{\Gamma_D}(\mathbf{\theta}) =\sum_{n=1}^N\sum_{p=1}^{P_\gamma} \left|\llbracket \mathcal{G}_{\mathbf{\theta}} (a^n)\rrbracket (\mathbf{x}_{n,p}^{\gamma}) - g_D(\mathbf{x}_{n,p}^{\gamma})\right|^2
=\sum_{n=1}^N\sum_{p=1}^{P_\gamma} \left|\mathcal{G}_{\mathbf{\theta}}^2(a^n)(\mathbf{x}_{n,p}^{\gamma}) - \mathcal{G}_{\mathbf{\theta}}^1(a^n)(\mathbf{x}_{n,p}^{\gamma})-g_D(\mathbf{x}_{n,p}^{\gamma})\right|^2,
\end{equation*}
and
\begin{equation*}
\begin{aligned}
L_{\Gamma_N}(\mathbf{\theta})&=\sum_{n=1}^N\sum_{p=1}^{P_\gamma} \left| \llbracket a^n \nabla \mathcal{G}_{\mathbf{\theta}}(a^n)\cdot \mathbf{n} \rrbracket (\mathbf{x}_{n,p}^{\gamma})-g_N(\mathbf{x}_{n,p}^{\gamma})\right|^2\\
&=\sum_{n=1}^N\sum_{p=1}^{P_\gamma} \left| a_2^n(\mathbf{x})\nabla \mathcal{G}_{\mathbf{\theta}}^2(a^n)(\mathbf{x}_{n,p}^{\gamma})\cdot\mathbf{n} -  a_1^n(\mathbf{x})\nabla \mathcal{G}_{\mathbf{\theta}}^1(a^n)(\mathbf{x}_{n,p}^{\gamma})\cdot\mathbf{n}-g_N(\mathbf{x}_{n,p}^{\gamma})\right|^2.
\end{aligned}
\end{equation*}
Here, $\{\mathbf{x}_{n,p}^{\gamma}\}_{p=1}^{P_\gamma}$ represents a set of training points sampled from the interface $\Gamma$ for the $n$-th input function. Then, the physics-informed loss function for IONet is formulated as follows:
\begin{equation}
\label{eq: empirical loss function}
    L_{physics}(\mathbf{\theta})=\lambda_1L_{r_1}(\mathbf{\theta})
    +\lambda_2L_{r_2}(\mathbf{\theta})
    +\lambda_3L_{\Gamma}(\mathbf{\theta})
    +\lambda_4L_b(\mathbf{\theta}),
\end{equation}
where $L_{r_i}(\mathbf{\theta})$  \eqref{eq: loss term Li} with $i=1,2$ are to approximately restrict the IONet output function to obey the given governing PDE (\ref{eq:1A}), while $L_b(\mathbf{\theta})$ and $L_{\Gamma}(\mathbf{\theta})$ penalize IONet for violating the boundary condition (\ref{eq:1D}) and the interface conditions (\ref{eq:1B}) and (\ref{eq:1C}), respectively. Such physics-informed loss function of IONet is schematically depicted on the right side of Fig. \ref{fig: ino neural architecture}. By incorporating physics constraints to ensure that the IONet output function aligns with the given interface PDE \eqref{eq: interface pde}, the proposed IONet can effectively learn the solution operator for parametric interface problems, even in the absence of labeled training data (excluding boundary and interface conditions). If both data and PDEs are available, we combine the loss functions \eqref{eq: ddloss of ino} and \eqref{eq: empirical loss function} and minimize the following composite loss function to obtain the parameter $\mathbf{\theta}$ of IONet:
\begin{equation}
L(\mathbf{\theta}) = \lambda_p L_{physics}(\mathbf{\theta})+\lambda_o L_{operator}(\mathbf{\theta}).
\label{eq: loss function con}
\end{equation}

\section{Numerical Results\label{numerical examples}}
In this section, the proposed IONet is tested on a range of parametric elliptic interface problems. Throughout all benchmarks, the branch nets and the trunk nets are FNNs. Particularly when $L_{physics}$ is included in (\ref{eq: loss function con}), i.e., $\lambda_p\neq 0$, they are FNNs with smooth activation function Tanh, due to the necessity for high-order derivatives. All operator network models are trained via stochastic gradient descent using Adam optimizer with default settings. The learning rate is set to exponential decay with a decay-rate of 0.95 per $\# \text{Epochs}/100$ iterations, where $\# \text{Epochs}$ denotes the maximum number of optimization iterations. Unless otherwise specified, the training data points used to evaluate the loss function are scattered points randomly sampled in the computational domain, while those used to evaluate the output solution of neural models are equidistant grid points.
After training, the average relative $L^2$ error between the reference solution operator $\mathcal{G}$ and the numerical solution operator $\mathcal{G}_{\theta}$ is measured as
\begin{equation*}
L^2(\mathcal{G}, \mathcal{G}_{\theta}) = \frac{1}{N} \sum_{n=1}^{N} \sqrt{\frac{\int_{\Omega} \left|\mathcal{G}(a^n_{test})(x)-\mathcal{G}_{\theta}(a^n_{test})(x)\right|^2 dx}{\int_{\Omega} \left|\mathcal{G}(a^n_{test})(x)\right|^2 dx}},
\end{equation*}
where $N$ denotes the number of test input functions $\{a^n_{test}\}_{n=1}^N$, and the integration is computed by the Monte Carlo method. For simplicity, IONet using the loss function \eqref{eq: loss function con} with $\lambda_p=1$ and $\lambda_o=0$ is referred to as ``PI-IONet", while IONet using the loss function \eqref{eq: loss function con} with $\lambda_p=0$ and $\lambda_o=1$ is denoted as ``DD-IONet".
All experiments are tested on one NVIDIA Tesla V100 GPU.

\subsection{Parametric elliptic interface problems in one dimension}

\hypertarget{example1}{\textbf{Example} 1}.
As the first example,  we investigate the effectiveness of the proposed method in handling non-zero interface conditions in the elliptic interface problem \eqref{eq: interface pde}, defined on the interval $\Omega=[0,1]$ with an interface point at $x_\gamma=0.5$:
\begin{equation}  \label{eq: 1d coefficient}
\begin{aligned}
 -\nabla\cdot(a(x)\nabla u(x)) &=  0,\quad x \  \in\  \Omega,\\
g_D(x_\gamma)=1, g_N(x_\gamma)&=0,
 \\
u(0)=1,\ u(1)&=0.
\end{aligned}
\end{equation}
Here, our goal is to learn a solution operator $\mathcal{G}$ that maps the discontinuous coefficient function $a(x)$ to the latent solution $u(x)$ that is explicitly discontinuous across the interface.

To make the input function
\begin{equation*}
a(x) = \left\{\begin{aligned}
& a_1(x),\quad x \in \Omega_1:=[0,0.5]\\
& a_2(x),\quad x \in \Omega_2:=(0.5,1]\\
\end{aligned}\right.
\end{equation*}
strictly positive, we let $a_i(x) = \Tilde{a}_i(x)-\min_x\Tilde{a}_i(x)+1$ with $i=1,2$, where $\Tilde{a}_i(x)$ is randomly sampled from a mean-zero Gaussian random field (GRF) with a radial basis function (RBF) kernel
\begin{equation*}
k_{l}(x_1, x_2)=\text{exp}\left(-\frac{\|x_1-x_2\|^2}{2l^2}\right)
\end{equation*}
using a length scale $l=0.25$ (see the left panel of Fig. \ref{fig: example 1d coef results} for an illustration).
We randomly sample $10,000$ and $1,000$ input functions $a(x)$ for training and testing, respectively. The sensor of the input function consists of $100$ equidistant grid points in the interval $[0,1]$. For each input function, we solve Eq. \eqref{eq: 1d coefficient} on a uniform mesh of size $1000$ using the matched interface and boundary (MIB) method with second-order accuracy \cite{xia2011mib} to obtain the reference solution and paired input-output training data. The test error of all neural models is measured on the same mesh of size $1000$.

In this example, we  investigate the performance of DD-IONet and PI-IONet as well as two state-of-the-art neural models, namely vanilla DeepONet (DD-DeepONet) \cite{lu2021learning} and physics-informed DeepONet (PI-DeepONet) \cite{wang2021learning}, in solving Eq. \eqref{eq: 1d coefficient} with variable coefficients $a(x)$. For PI-DeepONet, the neural network $\mathcal{G}_{\mathbf{\theta}}$ is trained by minimizing a physics-informed loss function of the form $L_{physics} (\mathbf{\theta}) =  \lambda_rL_{r} (\mathbf{\theta}) + \lambda_3L_{\Gamma} (\mathbf{\theta}) + \lambda_4 L_{b} (\mathbf{\theta})$, where $L_{\Gamma} =  L_{\Gamma_D} + L_{\Gamma_N}$, $\lambda_r = (\lambda_1+\lambda_2)/2$, and weights $\lambda_i$ with $i=1,2,3$ and 4 are those in PI-IONet. Note that the output function space of DeepONet is a subset of the space of continuous functions, implying  that $\llbracket \mathcal{G}_{\mathbf{\theta}}(a)(x_\gamma)\rrbracket=0$ holds for any input function $a$. Hence, we approximate the interface loss functions $L_{\Gamma_D}$ and $L_{\Gamma_N}$ in PI-DeepONet by difference schemes
\begin{equation*}
    L_{\Gamma_D} (\mathbf{\theta}) = \sum_{n=1}^N\Big|\mathcal{G}_{\mathbf{\theta}}(a^n) (x^\gamma+\epsilon)-\mathcal{G}_{\mathbf{\theta}}(a^n) (x^\gamma-\epsilon)-1\Big|^2
\end{equation*}
and
\begin{equation}
\begin{aligned}
L_{\Gamma_N}(\mathbf{\theta})&=\sum_{n=1}^N\Big| \llbracket a^n \nabla \mathcal{G}_{\mathbf{\theta}}(a^n)\cdot \mathbf{n} \rrbracket (x^{\gamma})\Big|^2\\
&=\sum_{n=1}^N \left|{a_2^n(x^{\gamma})}\frac{\mathcal{G}_{\mathbf{\theta}}(a^n)(x^{\gamma}+\epsilon)-\mathcal{G}_{\mathbf{\theta}}(a^n)(x^{\gamma})}{\epsilon}
 - {a_1^n(x^{\gamma})}\frac{\mathcal{G}_{\mathbf{\theta}}(a^n)(x^{\gamma})-\mathcal{G}_{\mathbf{\theta}}(a^n)(x^{\gamma}-\epsilon)}{\epsilon}
 \right|^2
\end{aligned}
\label{eq: pi-deeponet approximate interface condition n}
\end{equation}
with $\epsilon=10^{-5}$ in practice. In all cases, the neural networks are trained after $4\times 10^4$ parameter updates. The network architecture details and training costs are shown in Table  \ref{tab: 1d variable coef}.

\begin{figure}[ht]
	\centering
	\scalebox{0.6}{\includegraphics{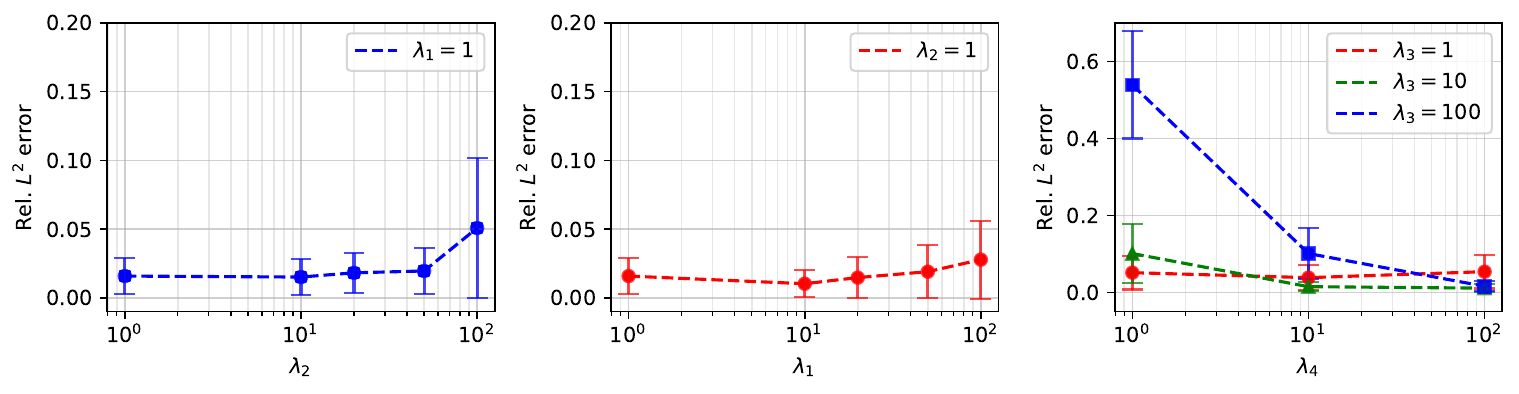}}
\caption{The mean and one standard deviation of relative $L^2$ error for PI-IONet with different weights in the physics-informed loss function \eqref{eq: empirical loss function}. Left: $\lambda_1=1$ and $\lambda_3=\lambda_4=100$. Middle: $\lambda_2=1$ and $\lambda_3=\lambda_4=100$. Right:   $\lambda_1=\lambda_2=1$.}
\label{fig: example 1d coef results lam}
\end{figure}

We first investigate the effect of the weights in loss function \eqref{eq: empirical loss function} on the accuracy of PI-IONet. It can be observed in Fig. \ref{fig: example 1d coef results lam} that the weights of interface and boundary loss terms (i.e. $\lambda_3$ and $\lambda_4 $) have more significant impacts on the numerical results compared with those of PDE residuals (i.e., $\lambda_1$ and $\lambda_2$). This phenomenon could be attributed to the fact that the investigated interface problem exhibits low-contrast. In the following, we fix the weights in physics-informed loss function of PI-IONet as $\lambda_1=\lambda_2=1$, $\lambda_3=10$, and $\lambda_4=100$.

\begin{table}[htbp]
\centering
\fontsize{10}{8}\selectfont
\begin{threeparttable}
\caption{Test errors and training costs for DD-DeepONet, DD-IONet, PI-DeepONet, and PI-IONet. The error corresponds to the relative $L^2$ error, recorded in the form of mean $\pm$ standard deviation based on all test input functions in Example \protect\hyperlink{example1} {1}.
}
\begin{tabular}{c|cccccc}
\toprule
 Models &Activation&Depth & Width  & \#Parameters&    $L^2(\mathcal{G}_{\theta}, \mathcal{G})$ & Training time (hours)\\
\midrule
DD-DeepONet & ReLU &5 & 140& 172K & 9.82e-2$\pm$1.71e-2  &0.10 \\
DD-IONet & ReLU &5 & 100& 172K & 3.95e-3$\pm$1.43e-3 & 0.17  \\
\midrule
PI-DeepONet &Tanh & 5& 140 & 172K & 4.70e-1$\pm$1.02e-1 & 0.36   \\
PI-IONet& Tanh& 5& 100 & 172K & 8.30e-3 $\pm$7.92e-3 & 0.44 \\
\bottomrule
\end{tabular}
\label{tab: 1d variable coef}
\end{threeparttable}
\end{table}

\begin{figure}[ht]
	\centering
	\scalebox{0.6}{\includegraphics{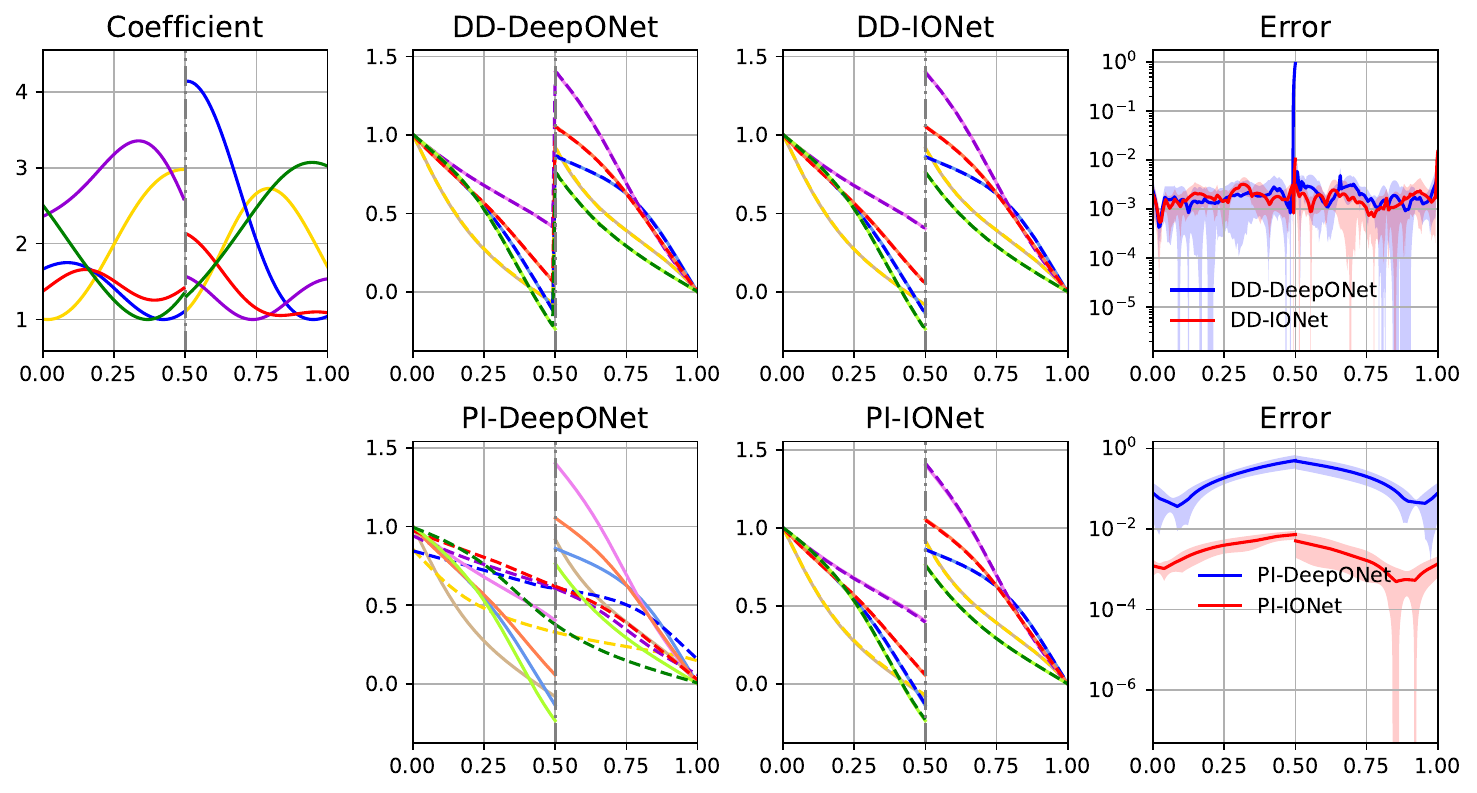}}
	\caption{Left column: Five input functions randomly selected from the test set (distinguished by different colors). Second and third columns: The reference solutions (solid lines) versus the numerical solutions (dashed lines) of DD-DeepONet, DD-IONet, PI-DeepONet, and PI-IONet. Fourth column: The mean and one standard deviation of the numerical solutions, averaged over these 5 test examples.}
\label{fig: example 1d coef results}
\end{figure}

Table \ref{tab: 1d variable coef} reports the relative $L^2$ errors between the reference solution and the numerical solution for DD-DeepONet, DD-IONet, PI-DeepONet, and PI-IONet. Under different training frameworks, it can be observed that the accuracy of DD-IONet is significantly superior to that of DD-DeepONet with the same paired input-output training data, while the error of PI-DeepONet is about 50 times that of PI-IONet when trained by minimizing the physics-informed loss functions. Furthermore, one can also see that DD-IONet outperforms PI-IONet, as the latter is trained without relying on any high-quality paired training data but instead through solving a highly complex optimization problem involving derivatives. In terms of the training time for neural models, as shown in this table, training physics-informed models (PI-DeepONet and PI-IONet) generally takes longer than training data-driven models (DD-DeepONet and DD-IONet). It is mainly due to the fact that physics-informed models require the computation of the PDE and interface residuals via automatic differentiation, and the loss terms are computed in a serial manner.

Fig. \ref{fig: example 1d coef results} shows a comparison between the reference and the numerical solutions for five randomly sampled input functions from the test dataset. The second column gives the numerical results of DD-DeepONet and PI-DeepONet. It can be observed that the numerical solution of DD-DeepONet agrees well with the reference solution away from the interface ($x_\Gamma=0.5$), but there are large errors near the interface (refer to the error plot in the first row), while PI-DeepONet fails to yield accurate results. The main reason is due to the fact that the output function space of DeepONet is a subset of continuous function space, which limits its effectiveness in capturing discontinuities in solution functions to interface problems.
In addition, we remark that DD-DeepONet is trained with high-quality paired training data using least squares regression, enabling the continuous output function of DeepONet to approximate the discontinuous output of the target operator point by point. However, PI-DeepONet employs a physics-informed loss function involving derivatives and the loss of interface conditions is approximated using difference schemes, exacerbating the challenge of handling discontinuities. This results in the numerical accuracy of DD-DeepONet being superior to that of PI-DeepONet.
The numerical results for DD-IONet and PI-IONet are displayed in the third column, where it can be seen that the numerical solutions naturally maintain the discontinuous nature of the numerical solution at the interface, demonstrating excellent agreement with the reference solutions (see the error plot in the fourth column). These numerical results show that IONet is more adaptable to the irregularities in the input function and the solution, which exhibits a greater ability to represent the solutions of interface problems compared with the conventional DeepONet architecture.

Next, we show that IONet is able to be extended to finite number of subdomains (greater than two). Consider Eq. \eqref{eq: 1d coefficient} with three subdomains, i.e.,
\begin{equation}  \label{eq: 1d coefficient multi-subdomains}
\begin{aligned}
 -\nabla\cdot(a(x)\nabla u(x)) &=  0,\quad x \  \in\  \Omega,\\
g_D(x_{\gamma_1})=1, g_N(x_{\gamma_1})&=0,\\
g_D(x_{\gamma_2})=-\frac{1}{2}, g_N(x_{\gamma_2})&=0,\\
u(0)=1,\ u(1)&=0,
\end{aligned}
\end{equation}
where $x_{\gamma_1}=0.3$ and $x_{\gamma_1}=0.7$.

\begin{figure}[ht]
	\centering
	\scalebox{0.75}{\includegraphics{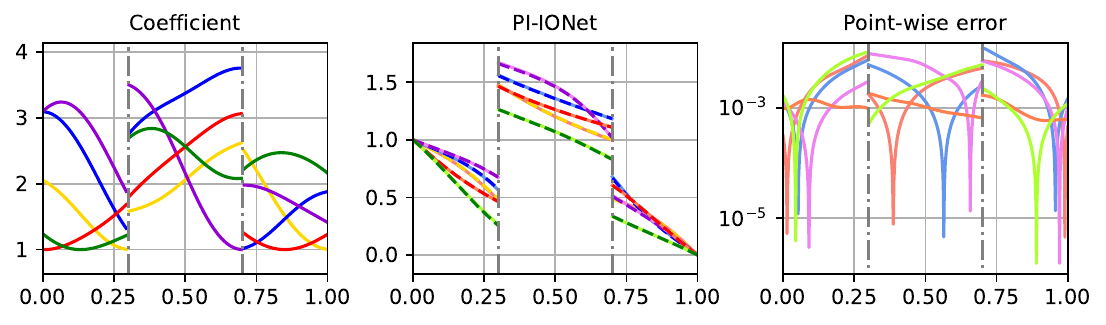}}
	\caption{Left: Five input functions randomly selected from the test set (distinguished by different colors). Middle: Reference solutions (solid lines) versus the numerical solutions (dashed lines) of PI-IONet. Right: Absolute point-wise errors over the whole domain. The gray point-dashed lines represent the location of the interfaces.}
\label{fig: example 1d coef results three}
\end{figure}
We approximate the solution operator $\mathcal{G}$ of Eq. \eqref{eq: 1d coefficient multi-subdomains} by PI-IONet. The chosen model architecture and other hyperparameters remain consistent with those used for two subdomain scenarios, except for the use of three branch nets and three trunk nets. Fig. \ref{fig: example 1d coef results three} shows the numerical results of PI-IONet. The final mean of relative $L^2$ error over five test input functions in this figure is measured at $3.83\times10^{-3}$. These results demonstrate that IONet also performs well in multi-subdomain scenarios.

\hypertarget{example2}{\textbf{Example} 2}. Next, we consider Eq. \eqref{eq: interface pde} with zero
interface conditions, defined on $\Omega=[0,1]$ with an interface point at $x^{\gamma}=0.5$:
\begin{equation}  \label{eq: 1d source term}
\begin{aligned}
 -\nabla\cdot\left(a(x)\nabla u(x)\right) + b(x)u(x) &=  f(x),\quad x \  \in\  \Omega,\\
g_D(x_\gamma)
=0,\ g_N(x_\gamma)&=0, \\
u(0)= h_0,\ u(1)&= h_1.
\end{aligned}
\end{equation}

From the setting of the interface conditions, we know that the latent solution $u$ to Eq. \eqref{eq: 1d source term} is continuous,  while its derivative may be discontinuous on the interface due to the different values of $a$ in different subdomains. In this example, our goal is to learn an operator $\mathcal{G}$ that maps the source term $f(x)$, boundary condition $h(x)$, and coefficients $a(x)$ and $b(x)$ to the solution $u$, i.e.,
\begin{equation}
 \mathcal{G}: \left(f(x),h(x),a(x), b(x)\right)\rightarrow u(x).
 \label{eq: 1d multi-input solution operator}
\end{equation}

To obtain the dataset, we randomly sample $10,000$ and $1,000$ sets of input functions $(f, h, a, b)$ for training and testing, respectively. Specifically, the input functions $f_1:=f|_{\Omega_1}$ and $f_2:=f|_{\Omega_2}$ are independently sampled from a zero-mean GRF with length scales $l_1=0.2$ and $l_1=0.1$, respectively. The coefficients $b_i:=b|_{\Omega_i}$ with $i=1,2$ are independently sampled via $b(x) = \Tilde{b}(x)-\min_x\Tilde{b}(x)+1$, where $\Tilde{b}(x)$ is randomly sampled from a zero-mean GRF with length scale $l=0.25$. The coefficient $a(x)$ is modeled as a piece-wise constant function, with $a|_{\Omega_1}=a_1$ and $a|_{\Omega_2} =a_2$, where $a_1$ and $a_2$ are sampled from uniform distributions over the intervals $[0.5, 1]$ and $[2, 3]$, respectively. Additionally, we sample $h_0$ and $h_1$ from uniform distributions over the intervals $[-0.1,0]$ and $[0,0.1]$, respectively, to set the variable boundary conditions $h(0) =h_0$ and $h(1)=h_1$. In this setting, it is easy to verify the existence and uniqueness of the solution to problem \eqref{eq: 1d source term}. For each set of input, the reference solution $u(x)$ is obtained by the MIB method with a uniform mesh of size $1000$. The test error of all neural models is measured on the same mesh of size $1000$.

\begin{figure}[h]
	\centering
	\scalebox{0.60}{\includegraphics{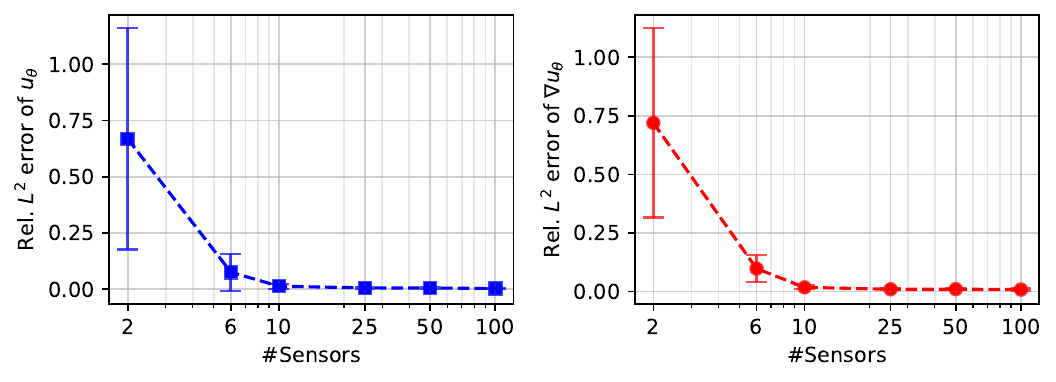}}
	\caption{The mean and one standard deviation of test $L^2$ errors for PI-IONet with varying numbers of sensors.}
 \label{fig: example2 error vs sensor}
\end{figure}

In this case, we investigate the performance of PI-IONet and PI-DeepONet in solving Eq. \eqref{eq: 1d source term}
without any paired input-output training data. In order to adapt to multiple input functions, for PI-IONet, we first divide the set of sensors $\{y_i\}_{i=1}^{n}$ into two subsets $\{y_i^1\}_{i=1}^{n_1}$ and $\{y^2_i\}_{i=1}^{n_2}$ with $n_1+n_2=n$ according to the interface. Then we employ $8$ branch nets to extract latent representations of the $4$ input functions on the corresponding sub-domain. Specifically, the inputs for branch nets are $[f(y^1_1),\cdots, f(y^1_{n_1})]$, $[f(y^2_1),\cdots, f(y^2_{n_2})]$, $[b(y^1_1),\cdots,b(y^1_{n_1})]$, $[b(y^2_1),\cdots,b(y^2_{n_2})]$, $[a_1]$, $[a_2]$, $[h_0]$ and $[h_1]$, respectively. For PI-DeepONet, we employ the most direct approach by concatenating the four input functions together, i.e., $[f(y_1),\cdots , f(y_n), b(y_1), \cdots ,b(y_n),a_1,a_2,h_0,h_1]$, to serve as a single input for the branch net.

We first utilize the PI-IONet network architecture, employing branch and trunk nets consisting of 5-layer FNN with 70 units per layer, to approximate the solution operator \eqref{eq: 1d multi-input solution operator}. The neural network is trained by minimizing the loss function \eqref{eq: empirical loss function} with $\lambda_1=\lambda_2=1$, $\lambda_3=10$, and $\lambda_4=100$, over $4\times10^4$ iterations of optimization. Fig. \ref{fig: example2 error vs sensor} illustrates the variation in relative $L^2$ errors of the numerical solution and its derivative with respect to the number of sensors (denoted as $\#\text{Sensors}$). Note that, in each case, the set of sensors consists of equidistant grid points in the interval $[0,1]$. As we can see from this figure, the test errors generally decrease as $\#\text{Sensors}$ increases until it is sufficient to capture all the necessary frequency information for the input function.

\begin{table}[ht]
\centering
\fontsize{9}{7}\selectfont
\begin{threeparttable}
\caption{Test errors and training costs for PI-IONet and PI-DeepONet. The error corresponds to the relative $L^2$ error, recorded in the form of mean $\pm$ standard deviation based on all test input functions in Example   \protect\hyperlink{example2} {2}.}
\begin{tabular}{c|ccccccc}
\toprule
 Model &Activation&Depth & Width  & \#Parameters&  $L^2(\mathcal{G}_{\theta}, \mathcal{G})$ &  $L^2(\nabla\mathcal{G}_{\theta}, \nabla\mathcal{G})$ & Training time (hours)\\
 \midrule
PI-IONet & Tanh &5 & 70 & 213K & 3.48e-3$\pm$ 3.05e-3 &   8.84e-3$\pm$4.51e-3 &  0.76 \\
PI-DeepONet & Tanh & 5 & 151 & 214K & 2.96e-1$\pm$ 1.99e-1 &  3.06e-1$\pm$ 1.97e-1 &  0.30 \\
\bottomrule
\end{tabular}
\label{tab: example 1d rhs error}
\end{threeparttable}
\end{table}

\begin{figure}[htbp]
	\centering
	\scalebox{0.45}{\includegraphics{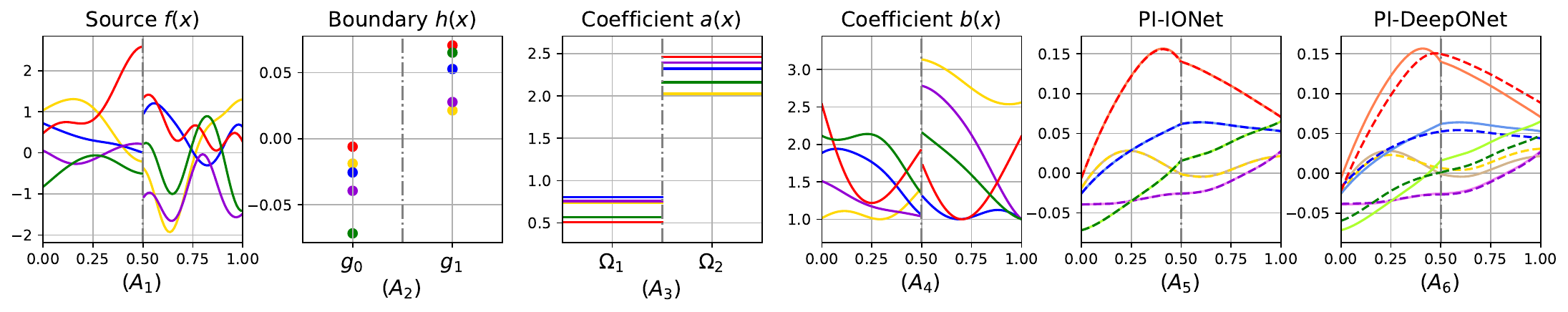}}
	\caption{An illustration of randomly sampled input functions (distinguished by different colors): Source term $f$ ($A_1$), boundary condition $h$ ($A_2$), coefficient $a$  ($A_3$) and coefficient $b$ ($A_4$); and the corresponding reference solutions (solid lines) obtained by MIB method and numerical solutions (dashed lines) obtained by PI-IONet ($A_5$) and PI-DeepONet ($A_6$). The gray point-dashed lines represent the location of the interface.}
\label{fig: example1 1d multi-inputs numerical results}
\end{figure}

In the following experiments, we keep $\#\text{Sensors}=100$, while keeping other hyperparameters consistent with those used in PI-IONet. Here, $L_{\Gamma_D}(\mathbf{\theta})=0$ holds for any output function of PI-DeepONet, and $L_{\Gamma_N}$ is approximated using Eq. \eqref{eq: pi-deeponet approximate interface condition n}.
Table \ref{tab: example 1d rhs error}
records the network structures and the test errors for two neural models. As can be seen from the table, the final relative $L^2$ errors of the numerical solutions and their derivatives obtained by PI-IONet can be of the order of $10^{-3}$. Some visualizations of the input function, the reference solution, and the numerical result are shown in Fig. \ref{fig: example1 1d multi-inputs numerical results}. As can be seen from the figure, the numerical solution obtained by PI-IONet aligns more consistently with the reference solution, even in the case where the ground truth is continuous. These results suggest that IONet is able to effectively handle multiple inputs for parametric interface problems.

\begin{figure}[h]
	\centering
	\scalebox{0.6}{\includegraphics{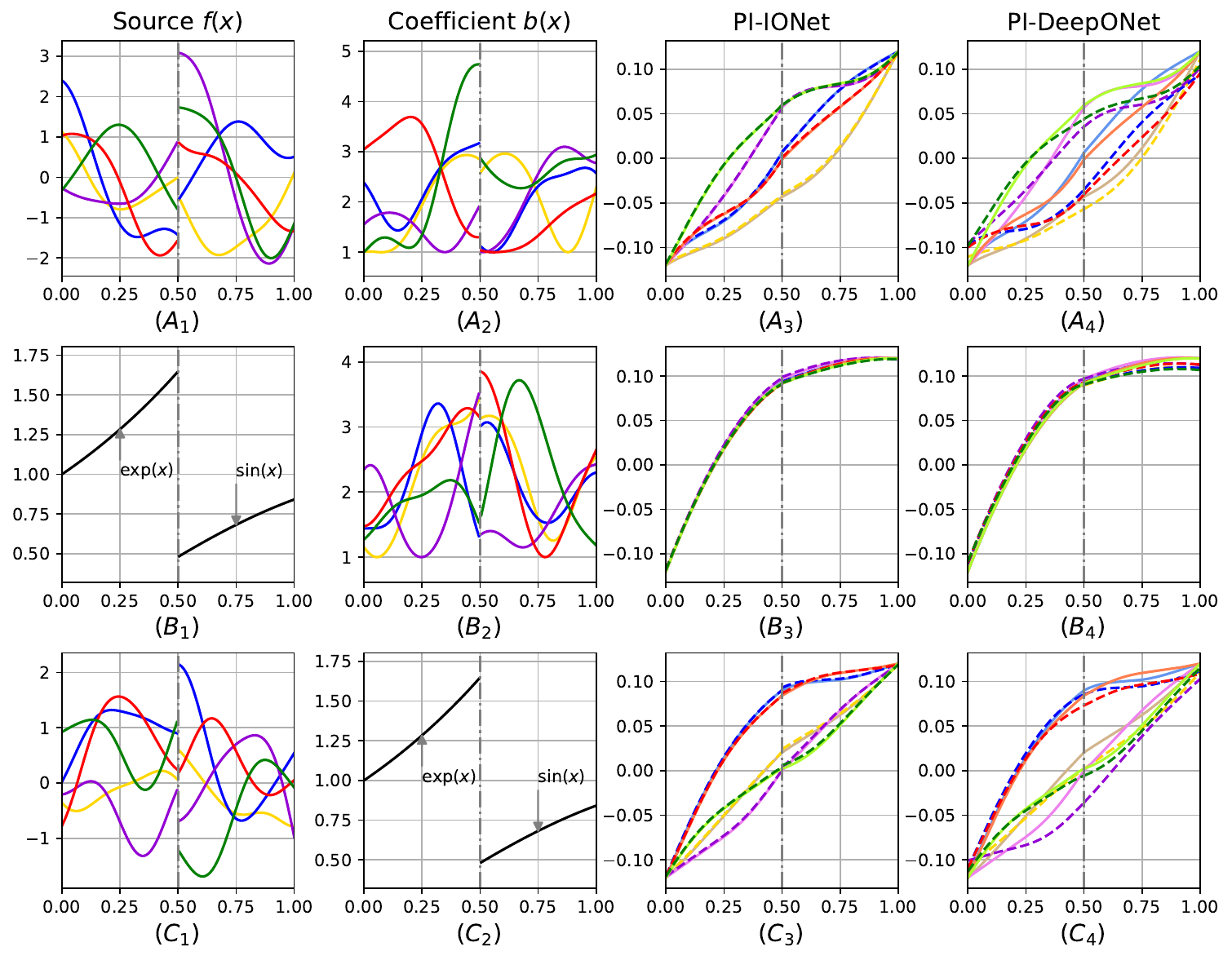}}
	\caption{
 Numerical results of PI-IONet and PI-DeepONet for out-of-distribution test input samples. Here, we keep $a_1=1.2$, $a_2=1.8$, $h_0=-0.12$ and $h_1=0.12$. The input functions in ($A_1$), ($A_2$), ($B_2$), and ($C_1$) are generated using GRF with length scale $l=0.15$. Input functions $f_1$ in ($B_1$) and $b_1$ in ($C_2$) are fixed as $\exp(x)$, while $f_2$ in ($B_1$) and $b_2$ in ($C_2$) are fixed as $\sin(x)$. The reference solutions (solid lines) and numerical solutions (dashed lines) obtained by PI-IONet and PI-DeepONet are shown in the third and fourth columns. The gray point-dashed lines represent the location of the interface.}
\label{fig: 1d rhs out-distribution prediction}
\end{figure}

It is remarked that other advantages of DeepONet also hold true for IONet, for example, the capability of providing accurate predictions for out-of-distribution test data \cite{wang2021learning, zhu2023reliable}. Fig. \ref{fig: 1d rhs out-distribution prediction} illustrates the numerical results of PI-IONet and PI-DeepONet for regenerated test input functions. In this study, we not only use GRF to generate the input functions (see the first row), but also include two certain functions, i.e., $\exp(x)$ and $\sin(x)$, as input functions (see the second and third rows). As shown in the third and fourth columns of this figure, both PI-IONet and PI-DeepONet integrate the constraints of physical laws as well as boundary conditions directly into model training, enabling the models to capture the fundamental behavior of the system. Remarkably, PI-IONet handles discontinuities in the input functions and output solutions across the interface more effectively, resulting in more accurate numerical results.
Specifically, the average relative $L^2$ errors of PI-IONet in the first to third rows are measured at $4.32\times10^{-3}$, $8.50\times10^{-3}$ and $2.01\times10^{-2}$, respectively. These results further underscore the robustness and generalization capability of PI-IONet.

\subsection{Parametric elliptic interface problems in two dimensions}
\hypertarget{example3}{\textbf{Example} 3}. To further investigate the capability of IONet, we consider a parametric interface problem \eqref{eq: interface pde} with a sharp and complicated interface $\Gamma$ which is given as
\begin{equation*}
\begin{aligned}
&x_1(\vartheta) = 0.65\text{cos}(\vartheta)^3, \\
&x_2(\vartheta) = 0.65\text{sin}(\vartheta)^3, \ 0\leq\vartheta\leq\pi.
\end{aligned}
\end{equation*}
Here, the source term $f$ is the input parameter of the target solution operator. In this example, we model the input function in the following way:
\begin{equation*}
f_i(\mathbf{x}):=f(\mathbf{x})\big|_{\Omega_i}=\frac{p^i_1}{[1+10(x_1^2+x_2^2)]^2}-\frac{p^i_2(x_1^2+x_2^2)}{[1+10(x_1^2+x_2^2)]^3},\ i=1,2,
\end{equation*}
where $(p^i_1,p^i_2)$ comes from $[50,100]\times[1550,1650]$. The computational domain is a regular square $\Omega=[-1,1]^2$ (see Fig. \ref{fig: example3 computational domain and sensor location.} for an illustration).
The coefficient $a(\mathbf{x})$ is a piece-wise constant, which is given by $a(\mathbf{x})|_{\Omega_1}=2$ and $a(\mathbf{x})|_{\Omega_2}=1$.
The interface conditions on $\Gamma$ are set as
\begin{equation*}
\begin{aligned}
&g_D(\mathbf{x}) = \frac{1}{1+10(x_1^2+x_2^2)}, \\
&g_N(\mathbf{x}) = 0,\\
\end{aligned}
\end{equation*}
and $h(\mathbf{x})$ on boundary $\partial\Omega$ is given as $$h(\mathbf{x})=\frac{2}{1+10(x_1^2+x_2^2)}.$$
One can observe that if we take $f_1(\mathbf{x})=f_2(\mathbf{x})$ with $(p^1_1,p^1_2)=(p^2_1,p^2_2)=(80,1600)$, then Eq. \eqref{eq: interface pde} has following exact solution \cite{wu2022interfaced}
\begin{equation*}
u(\mathbf{x})=\left\{
\begin{aligned}
&\frac{1}{1+10(x_1^2+x_2^2)}, \ &\text{in} \ \Omega_1, \\
&\frac{2}{1+10(x_1^2+x_2^2)}, \ &\text{in} \ \Omega_2. \\
\end{aligned}
\right.\end{equation*}
To this end, we randomly sample some input functions from the given data distribution, except for the case $(p^1_1,p^1_2)=(p^2_1,p^2_2)=(80,1600)$, which is reserved for testing purposes. The sensors for the input functions in the whole domain are equidistant grid points in the square $[-1, 1]^2$. Take advantage of being mesh-free, IONet can easily handle problems with irregular interfaces.

\begin{figure}[ht]
	\centering
	\scalebox{0.6}{\includegraphics{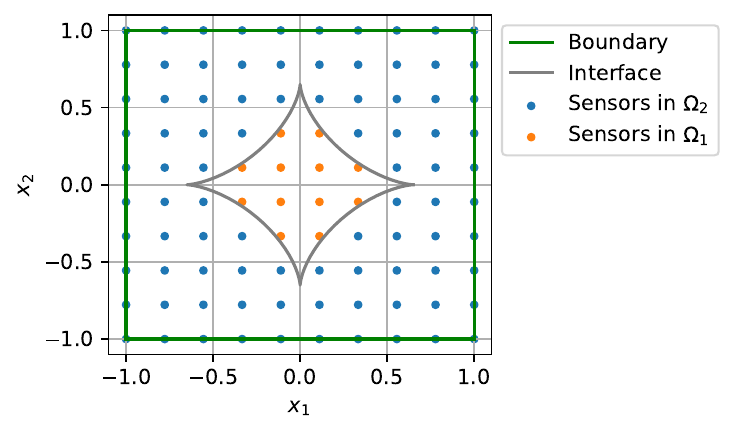}}
	\caption{Computational domain and sensor locations in Example \protect\hyperlink{example3}{3}. Here, the number of sensors is 100.}
\label{fig: example3 computational domain and sensor location.}
\end{figure}

\begin{figure}[ht]
	\centering
	\scalebox{0.52}{\includegraphics{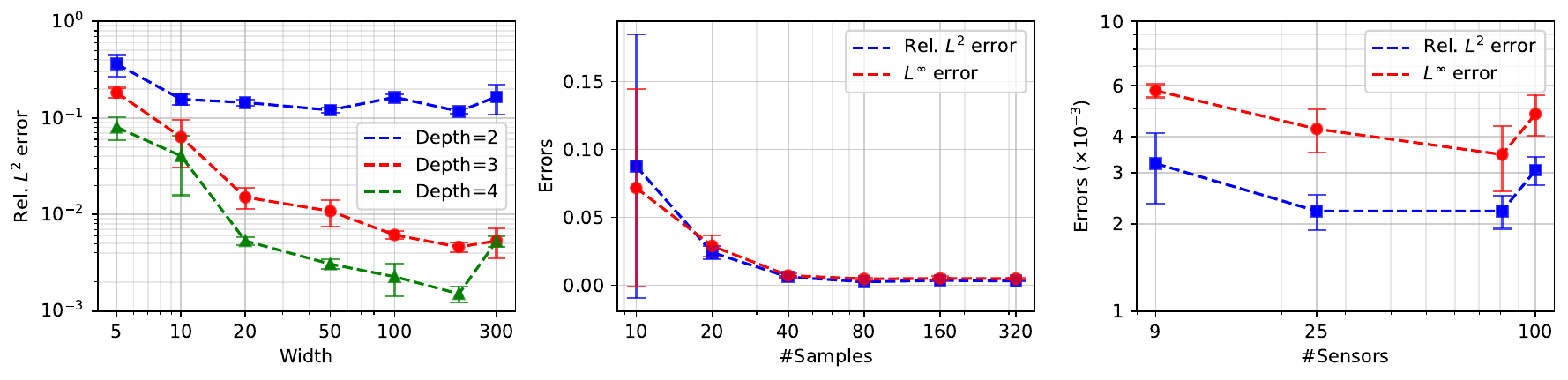}}
	\caption{The mean and one standard deviation of test errors of PI-IONet in Example \protect\hyperlink{example3}{3}. (Left)  Training PI-IONet using different depth and width of the network architecture, where $\#\text{Sensors}=100$ and $\#\text{Samples}=320$. (Middle) Training PI-IONet using different number of training samples, where $\text{width}=50$, $\text{depth}=4$ and $\#\text{Sensors}=100$. (Right) Training PI-IONet using different number of sensors, where $\text{width}=50$, $\text{depth}=4$ and $\#\text{Samples}=320$. These errors represent the average of 3 different runs corresponding to different set of training input functions and network initialization.}
\label{fig: example 3 error plot}
\end{figure}

\begin{figure}[htbp]
	\centering
	\scalebox{0.50}{\includegraphics{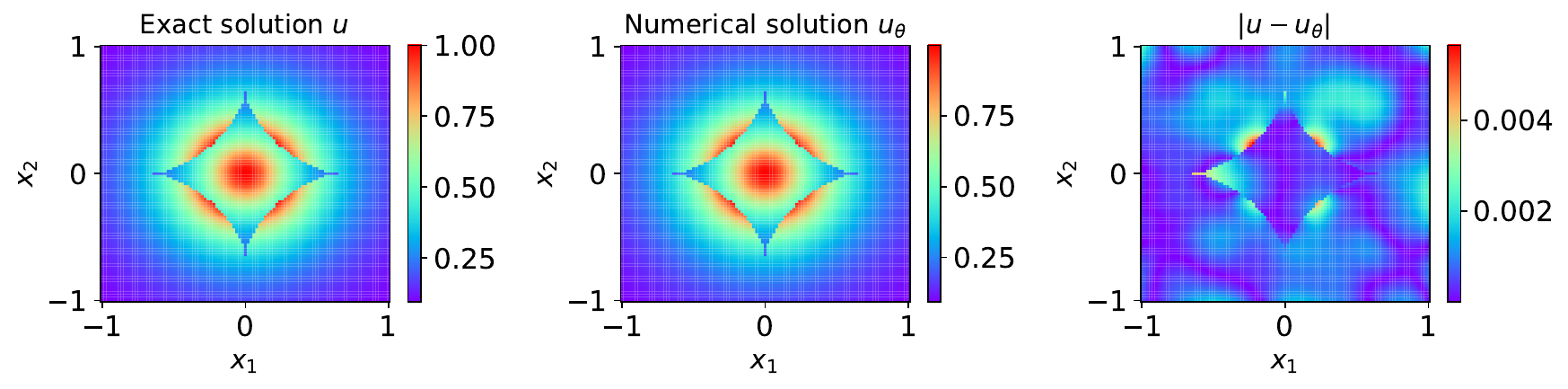}}
	\caption{The profile of the exact solution, the numerical solution obtained by PI-IONet, and the corresponding point-wise error are showcased from left to right. Here, $\text{width}=50$, $\text{depth}=4$, $\#\text{Sensors}=100$, and $\# \text{Samples}=320$. }
\label{fig: example3 solutions and error}
\end{figure}

In this study, we discuss the effects of the depth and width of the network, as well as the number of sensors and training samples,  on the performance of the PI-IONet. We train PI-IONet by minimizing the physics-informed loss function \eqref{eq: empirical loss function} with $\lambda_1=\lambda_2=1$, $\lambda_3=10$ and $\lambda_4=100$ for $4\times10^4$ iterations of optimization.
Fig. \ref{fig: example 3 error plot} illustrates the average error of three different runs of PI-IONet. Specifically, the left panel shows the relative $L^2$ error, i.e
$$
\frac{\|u-u_\theta\|_2}{\|u\|_2}=\sqrt{\frac{\sum_{i=1}^N\left(u(\mathbf{x}_i)-u_\theta(\mathbf{x}_i)\right)^2}{\sum_{i=1}^Nu(\mathbf{x}_i)^2}},
$$
between the exact solution $u$ and the PI-IONet solution $u_\theta$ measured at $N= 101\times 101$ test points over the whole domain, with varying depths (ranging from 2 to 4) and widths (ranging from 5 to 300) of the network architecture. We observe that increasing the expressiveness of the network leads to improved solution accuracy, eventually reaching a plain where the error reduction levels off. A similar trend is also observed on the middle panel as the number of training input samples (denoted as $\# \text{Samples}$) increases. These might be caused by optimization errors. Additionally, as the input function $f(\mathbf{x})$ in this case is controlled by four parameters (i.e., $p_1^i$ and $p_2^i$ with $i=1,2$), a small number of sensors could be enough to capture all necessary frequency information of the input functions. As depicted on the right panel of this figure, although the accuracy of PI-IONet does not significantly improve with an increase in $\#\text{Sensors}$, the average $L^\infty$ error and relative $L^2$ error of PI-IONet with different $\#\text{Sensors}$ can both reach the order of $10^{-3}$. Moreover, it can be seen from Fig. \ref{fig: example3 solutions and error}, the numerical solution is in excellent agreement with the exact solution, where the relative $L^2$ error is measured at $2.60\times10^{-3}$. These results demonstrate the consistent and reliable performance of IONet in generating accurate numerical results, even in scenarios where the interface is irregular.

\hypertarget{example4}{\textbf{Example} 4}. This example aims to investigate the performance of IONet in handling a two-dimensional parametric interface problem with variable boundary conditions and coefficients.
Computational domain is defined as $\Omega:= [0,1]^2$, and the interface is defined as $\Gamma:=\{\mathbf{x}:=(x_1,x_2) \ | \ x_2=0.5, \mathbf{x}\in \Omega\}$. Without losing generality, we defined $\Omega_1:=\{\mathbf{x} \ | \ x_2>0.5, \mathbf{x}\in \Omega\}$ and $\Omega_2:=\{\mathbf{x} \ | \ x_2<0.5, \mathbf{x}\in \Omega\}$ (see Fig. \ref{fig: example2 domain} for an illustration). Specifically, the interface problem takes the following specific form:
\begin{equation}
\begin{aligned}
 -\nabla\cdot(a\nabla u(\mathbf{x})) &=  0,\quad  &&\mathbf{x} \in\Omega,\\
u(\mathbf{x})&=h(\mathbf{x}),\quad &&\mathbf{x} \in \partial\Omega,
\label{eq: 2D hx}
\end{aligned}
\end{equation}
with interface conditions $g_D=0$ and $g_N=0$. Our goal here is to learn the solution operator that maps the coefficient $a(x)$ and the boundary condition $h(\mathbf{x})$ to the solution $u(\mathbf{x})$ to Eq. \eqref{eq: 2D hx}, i.e.,
\begin{equation*}
 \mathcal{G}: \left(h(x) ,a(x)\right)\rightarrow u(x).
\end{equation*}
\begin{figure}[ht]
	\centering
	\scalebox{0.35}{\includegraphics{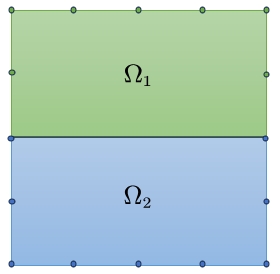}}
	\caption{Computational domain for Example \protect\hyperlink{example4}{4}. The black solid line indicates the interface location.}
\label{fig: example2 domain}
\end{figure}

In this example, we randomly sample $3,200$ and $100$ pairs of input functions $(h, a)$ for training and testing, respectively. Here, the boundary conditions $h_i(\mathbf{x}):=h(\mathbf{x})|_{
\partial\Omega\cap\Omega_i}$ with $i=1, 2$ are independently generated using GRF according to $h_i\sim \mu|_{\partial\Omega\cap\Omega_i}$, where $\mu\sim\mathcal{N}(0,10^{3}(-\Delta+100I)^{-4})$ with zero Neumann boundary conditions on the Laplacian\footnote{One common approach is to use a random number generator to sample from a normal distribution with zero mean and unit variance and then apply a spectral representation to generate the desired spatial correlation structure, see \cite{li2021fourier} for more details.}, while the coefficient $a(x)$ is modeled as a piece-wise constant function, where $a_1:=a|_{\Omega_1}$ and $a_2:=a_2|_{\Omega_2}$ are uniformly sampled from the intervals $[0.5, 1]$ and $[2,3]$, respectively. For each pair of input, the reference solution is obtained by the $\mathbb{P}_1$ Lagrangian finite element method\footnote{The implementation is based on the Fenics platform \cite{fenics2010anders}.} on a uniform mesh of 1025 by 1025, while the test error of the numerical solution is measured at its 65 by 65 submesh.

In this study, to accommodate two input functions, the IONet architecture consists of four branch nets and two trunk nets, while the DeepONet architecture consists of one branch net and one trunk net. Specifically, the inputs for branch nets in IONet are $[h(y^1_1),\cdots, h(y^1_{n_1})]$, $[h(y^2_1),\cdots, h(y^2_{n_2})]$, $[a_1]$ and $[a_2]$, while for DeepONet it is
$[h(y_1),\cdots, h(y_{n}), a_1, a_2]$, where $\{y_i\}_{i=1}^{n}=\{y_i^1\}_{i=1}^{n_1}\cup\{y^2_i\}_{i=1}^{n_2}$ is the set of sensors over the whole domain.

\begin{figure}[ht]
	\centering
	\scalebox{0.5}{\includegraphics{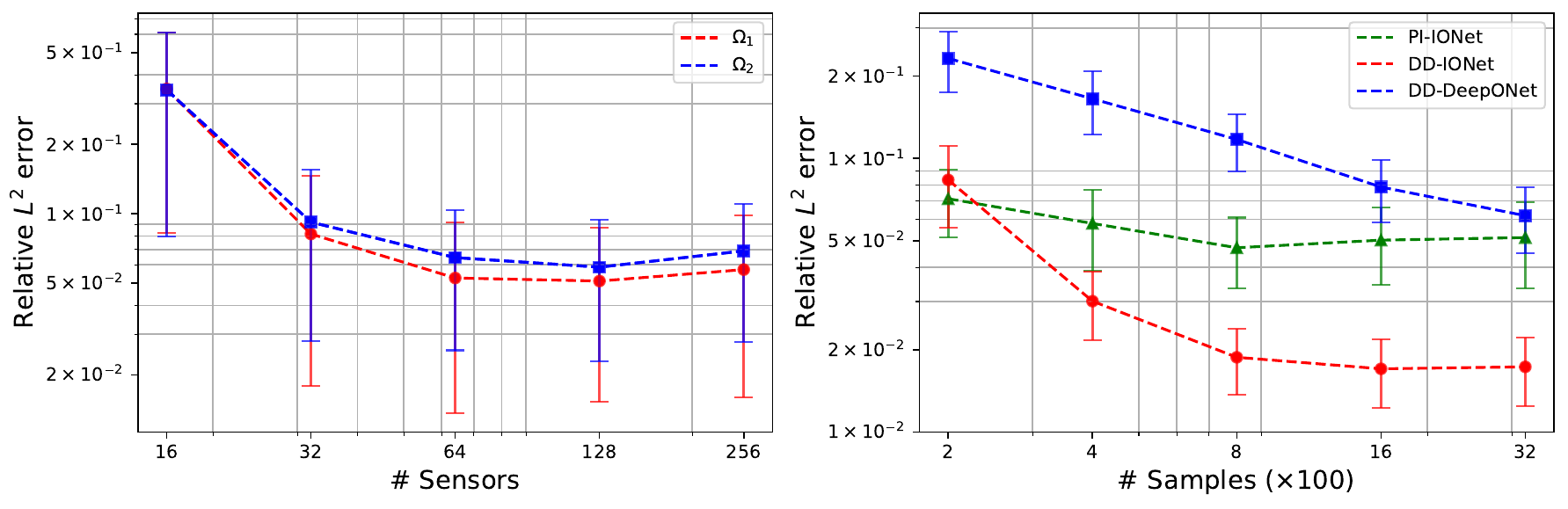}}
	\caption{Numerical results of Example \protect\hyperlink{example4}{4}. Left: the relative $L^2$ error of PI-IONet with respect to $\# \text{Sensors}$, where $\# \text{Samples}=1600$. Right: The relative $L^2$ error of PI-IONet, DD-IONet and DD-DeepONet with different number of training samples, where $\# \text{Sensors}=128$.}
\label{fig: example2 results}
\end{figure}

\begin{table}[ht]
\centering
\fontsize{10}{8}\selectfont
\begin{threeparttable}
\caption{Test errors and training costs for  PI-IONet,  DD-IONet and DD-DeepONet. The error corresponds to the relative $L^2$ error, recorded in the form of mean $\pm$ standard deviation based on all test input functions in Example  \protect\hyperlink{example4} {4}. Here,  $\# \text{Sensors}=128$ and $\# \text{Samples}=1600$.}
\begin{tabular}{c|cccccc}
\toprule
 Model &Activation& Depth & Width  & \#Parameters&  $L^2(\mathcal{G}_{\theta}, \mathcal{G})$ & Training time (hours) \\
 \midrule
PI-IONet & Tanh  &5 & 120 & 365K & 5.03e-2$\pm$1.58e-2 & 1.26 \\
DD-IONet & ReLU  & 5 & 120 & 365K &
1.70e-2$\pm$4.79e-3
&  0.47 \\
DD-DeepONet& ReLU &5& 206  & 368K& 7.56e-2$\pm$1.99e-2  & 0.20\\
\bottomrule
\end{tabular}
\label{tab: example4 compare}
\end{threeparttable}
\end{table}
\begin{figure}[htbp]
	\centering
	\scalebox{0.5}{\includegraphics{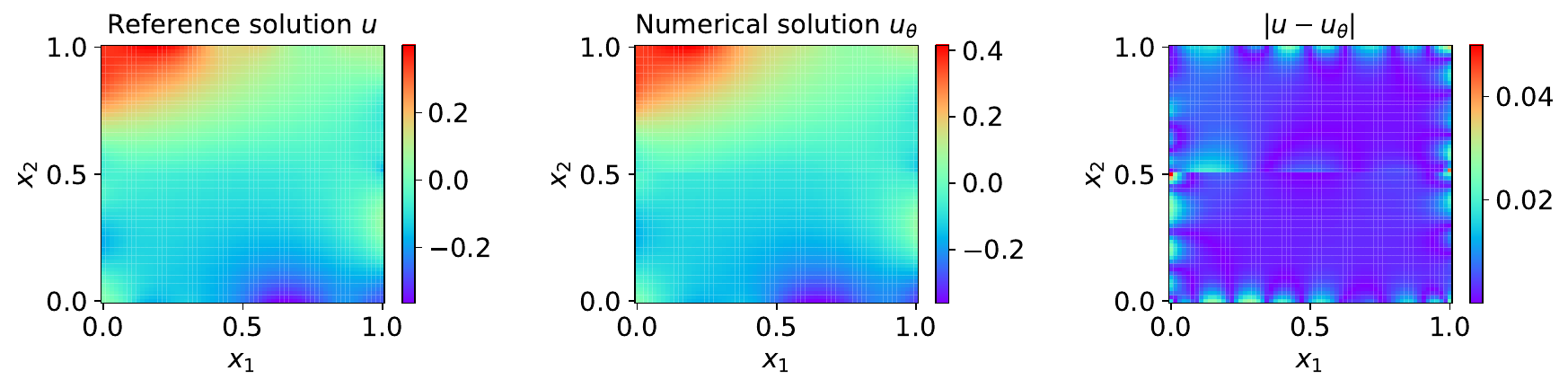}}
	\caption{The profile of the reference solution, the numerical solution obtained by PI-IONet, and the corresponding point-wise error in the whole domain for a representative example in the test dataset.}
\label{fig: 2d boundary FEM vs IONet}
\end{figure}

In the following experiments, the numerical results are recorded after $1\times10^{5}$ optimization iterations. The detailed network sizes are provided in Table \ref{tab: example4 compare}. Fig. \ref{fig: example2 results} illustrates the accuracy of PI-IONet, DD-IONet and DD-DeepONet with respect to the number of sensors or training samples.
Specifically, the left panel of this figure depicts the variation of the relative $L^2$ error for PI-IONet, which is trained by minimizing the physics-informed loss function \ref{eq: empirical loss function} with $\lambda_1=\lambda_2=1$, $\lambda_3=10$ and $\lambda_4=100$. It can be seen that the relative $L^2$ error measured at two subdomains decreases rapidly when $\# \text{Sensors}$ is less than 64; however, it tends to level off as $\# \text{Sensors}$ is further augmented due to other factors such as optimization errors and generalization errors. In addition, the errors in the two subdomains exhibit close proximity to one another, indicating the effectiveness of the proposed method in balancing errors across the subdomains.
In order to ascertain the effect of $\# \text{Samples}$ on the performance of PI-IONet, DD-IONet and DD-DeepONet, we maintain $\# \text{Sensors}= 128$. As illustrated in the right panel of Fig. \ref{fig: example2 results}, the relative $L^2$ errors tend to decrease with an increase in the number of samples. This observation aligns with the findings of Example \protect\hyperlink{example3} {3}, which concerns a parameterized interface problem with a single input source. It is noteworthy that the accuracy of PI-IONet is less sensitive to $\# \text{Samples}$ compared to DD-IONet and DD-DeepONet. For instance, when the number of training samples is limited to $\#\text{Samples}=200$, PI-IONet achieves the lowest relative $L^2$ error among the three models without any paired input-output measurements. Table \ref{tab: example4 compare} shows the relative $L^2$ error and the training cost of PI-IONet, DD-IONet and DD-DeepONet when $\# \text{Sensors}= 128$ and  $\# \text{Samples}= 1600$. It is observed that with a sufficiently large number of samples, the relative $L^2$ error of all three models can reach the order of $10^{-2}$, and DD-IONet outperforms DD-DeepONet. An illustration of a numerical solution obtained using PI-IONet is presented in Fig. \ref{fig: 2d boundary FEM vs IONet}, with the relative $L^2$ error measured at $4.13\times10^{-2}$. The numerical solution derived from PI-IONet exhibits consistency with the reference solution.

\subsection{Parametric elliptic interface problems in three dimensions}
\hypertarget{example5}{\textbf{Example} 5}.
To illustrate the capability of the proposed method for solving nonlinear interface problems, we consider the Poisson-Boltzmann equation (PBE), a prevalent implicit continuum model utilized in the estimation of biomolecular electrostatic potentials $\Phi(\mathbf{x})$. Similar equations occur in various applications, including electrochemistry and semiconductor physics. The molecule in the PBE is represented by a series of $N_m$ charges $q_i$ at positions $\mathbf{c}_i$, where $q_i=z_ie_c$, $z_i\in \mathbb{R}$, $i=1,\cdots , N_m$. Specifically, we choose a real molecule (PDBID: ADP) with $N_m$ = 39 atoms as an example. Without loss of generality, the molecule is translated from the average coordinate center of all atoms to the center of $\Omega=[-10,10]^3$. Then, in the special case of $1:1$ electrolyte, the PBE can be formulated for dimensionless potential $u(\mathbf{x}) = e_ck_B^{-1}T^{-1}\Phi(\mathbf{x})$ as follows:
\begin{equation}
\begin{aligned}
 - \nabla\cdot (\epsilon(\mathbf{x})\nabla u(\mathbf{x}))+\Bar{\kappa}^2(\mathbf{x})\sinh(u(\mathbf{x})) &=\alpha\sum_{i=1}^{N_m}z_i\delta(\mathbf{\mathbf{x}-\mathbf{c}_i}),  \; &&\mathbf{x}\in \Omega,\\
 \llbracket u(\mathbf{x})\rrbracket &= 0,\; &&\mathbf{x}\in\Gamma, \\
 \llbracket \epsilon(\mathbf{x})\frac{\partial u(\mathbf{x})}{\partial \mathbf{n}}\rrbracket  &= 0,\; &&\mathbf{x}\in\Gamma, \\
u(\mathbf{x})&=\frac{\alpha}{4\pi\epsilon(\mathbf{x})}\sum_{i=1}^{N_m}z_i\frac{e^{-\kappa\|\mathbf{x}-\mathbf{c}_i\|}}{\|\mathbf{x}-\mathbf{c}_i\|},\; &&\mathbf{x}\in\partial\Omega,
\end{aligned}
\label{eq: NPBE}
\end{equation}
where $\delta(\cdot)$ is the Dirac delta function, the permittivity $\epsilon(\mathbf{x})$ takes the values of $\epsilon_m\epsilon_0$ and $\epsilon_s\epsilon_0$ in the molecular region $\Omega_1$ and the solution region $\Omega_2$, respectively. The modified Debye-H$\ddot{\text{u}}$ckel takes the values $\Bar{\kappa}=0$ in $\Omega_1$ and $\Bar{\kappa}=\sqrt{\epsilon_m\epsilon_0}\kappa$ in $\Omega_2$, and constant $\alpha = \frac{e_c^2}{k_BT}$. Here, constants $\epsilon_0, e_c, \beta, \kappa$ and $T$ represent the vacuum dielectric constant,  fundamental charge, Boltzmann's constant, Debye-H$\ddot{\text{u}}$ckel constant and  absolute temperature, respectively. Our goal is to learn an operator $\mathcal{G}$ mapping from the permittivity $\epsilon(\mathbf{x})$ to the solution $u(\mathbf{x})$ to PBE.
Note that $\epsilon$ has a piece-wise constant nature, allowing us to directly utilize the function values as inputs for IONet without requiring sensor-based discretization.

To numerically solve PBE \eqref{eq: NPBE}, we use a solution decomposition scheme to overcome the singular difficulty caused by the Dirac delta distributions. Similar to our former work \cite{wu2022interfaced}, $u$ is decomposed as $u(\mathbf{x}) = \Bar{G}(\mathbf{x}) + \Bar{u}(\mathbf{x})$, where
\begin{equation*}
  \Bar{G}(\mathbf{x})=\frac{\alpha}{4\pi \epsilon_m\epsilon_0}\sum_{i=1}^{N_m}\frac{z_i}{\|\mathbf{x}-\mathbf{c}_i\|}
,\  \nabla  \Bar{G}(\mathbf{x})= -\frac{\alpha}{4\pi \epsilon_m\epsilon_0}\sum_{i=1}^{N_m}z_i\frac{\mathbf{x}-\mathbf{c}_i}{\|\mathbf{x}-\mathbf{c}_i\|^3}
\end{equation*}
are restricted to $\Omega_1$, and  $\Bar{u}(\mathbf{x})$ satisfies the following PDE:
\begin{equation*}
\begin{aligned}
-\nabla\cdot (\epsilon_m\epsilon_0\nabla \Bar{u}(\mathbf{x})) &=0, \ &&\mathbf{x} \in \Omega_1, \\
-\nabla\cdot(\epsilon_s\epsilon_0\nabla \Bar{u}(\mathbf{x}) ) + \Bar{\kappa}^2 \sinh(\Bar{u}(\mathbf{x}))  &=0, \ &&\mathbf{x} \in \Omega_2, \\
\llbracket \Bar{u}(\mathbf{x})\rrbracket  &= \Bar{G}(\mathbf{x}),\ &&\mathbf{x}\in\Gamma, \\
\llbracket \epsilon(\mathbf{x})\frac{\partial \Bar{u}(\mathbf{x})}{\partial \mathbf{n}}\rrbracket  &= \epsilon_m\epsilon_0\frac{\partial \Bar{G}(\mathbf{x})}{\partial \mathbf{n}},\ &&\mathbf{x}\in\Gamma, \\
\Bar{u}(\mathbf{x})&=\frac{\alpha}{4\pi\epsilon_s\epsilon_0}\sum_{i=1}^{N_m}z_i\frac{e^{-\kappa\|\mathbf{x}-\mathbf{c}_i\|}}{\|\mathbf{x}-\mathbf{c}_i\|}, \ &&\mathbf{x} \in\partial\Omega.
\end{aligned}
\end{equation*}

The training dataset comprises $1,000$ input functions (constants) uniformly and randomly sampled from the space $(\epsilon_m, \epsilon_s)\in[1,2]\times[80,100]$, while the test dataset is composed of equidistant grid points arranged in a $6 \times 6$ grid within this space.  For each test sample, we solved PBE \eqref{eq: NPBE} using piece-wise linear FEM \cite{ji2018finite} on an interface-fitted mesh to generate the reference solution. Specifically, the grid points of the FEM mesh consist of $1407$, $3403$, $2743$ and $2402$ points in $\Omega_1$, $\Omega_2$, $\Gamma$ and $\partial\Omega$, respectively. During the training phase, for each input function, the training points used to evaluate the loss function consist of a collection of one-tenth randomly sampled grid points, rather than a set of randomly sampled scatter points within the domain. This approach ensures consistency in the interface across different methods. Note that the unit outward normal vector for each point on the interface $\Gamma$ is approximated by taking the average of the outward normal directions of all elements that contain the corresponding point.

\begin{table}[htbp]
\centering
\fontsize{10}{8}\selectfont
\begin{threeparttable}
\caption{Test errors and training costs for PI-IONet, DD-IONet, and DD-DeepONet. The error corresponds to the relative $L^2$ error, recorded in the form of mean $\pm$ standard deviation based on all test input functions in Example \protect\hyperlink{example5} {5}.
}
\begin{tabular}{ccccccc}
\toprule
Model & Activation & Depth&  Width
& \#Parameters  &  $L^2(\mathcal{G}_{\theta}, \mathcal{G})$ & Training time (hours)\\
 \midrule
PI-IONet& Tanh & 5 & 150 & 364K & 1.24e-2$\pm$8.31e-5 &  0.51 \\
DD-IONet & ReLU & 5 & 150 & 364K &  4.33e-3$\pm$1.61e-3 &0.17 \\
DD-DeepONet& ReLU & 5  & 215 & 373K  & 1.54e-2$\pm$3.80e-4   &  0.10
 \\
\bottomrule
\label{tab: 3d pb}
\end{tabular}
\end{threeparttable}
\end{table}

\begin{figure}[htbp]
\centering
\subfigure[Surface potential obtained by  FEM.]{
\begin{minipage}{7cm}
\centering
\includegraphics[scale=0.35]{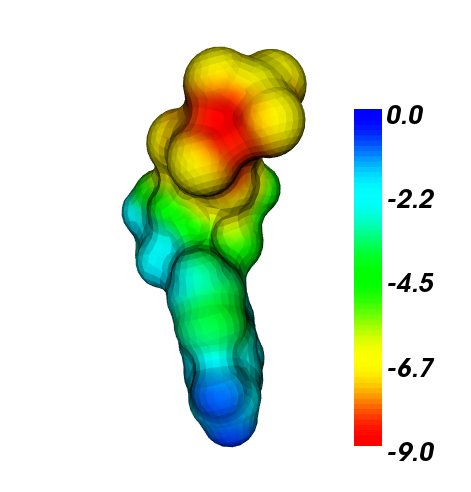}
\end{minipage}
}
\subfigure[Surface potential obtained by PI-IONet.]{
\begin{minipage}{7cm}
\centering
\includegraphics[scale=0.35]{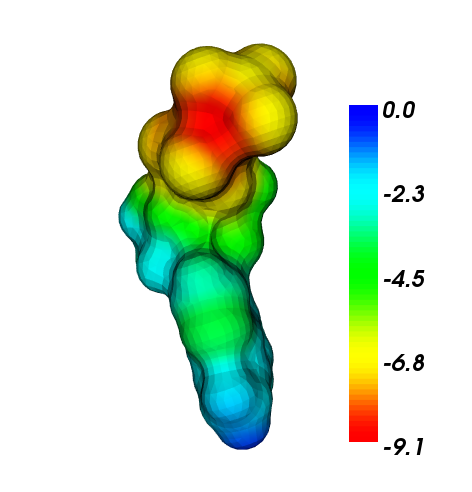}
\end{minipage}
}
\caption{Surface potentials for protein ADP. Here, $\epsilon_m=2$ and $\epsilon_s =80$.}
\label{fig: adp plot}
\end{figure}

In this example, we approximate the solution operator of PBE \eqref{eq: NPBE} using PI-IONet, DD-IONet and DD-DeepONet. Herein, the weights in the loss function \eqref{eq: empirical loss function} of PI-IONet are set as $\lambda_1=\lambda_2=\lambda_3=1$ and $\lambda_4=100$. Table \ref{tab: 3d pb} records test errors measured at 4810 grid points in $\Omega_1$ and $\Omega_2$ and training costs of these three models after $5 \times 10^4$ parameter updates. It can be observed that DD-IONet achieves the lowest relative $L^2$ error. Additionally, despite lacking any paired input-output measurements, except for the boundary conditions, the error accuracy of PI-IONet is comparable to those of DD-DeepONet, albeit with slightly higher training costs. A visual comparison of the reference and predicted surface potentials of the protein ADP is shown in Fig. \ref{fig: adp plot}. These findings further emphasize the capability of IONet handle parametric interface problems within irregular interface. While our current work demonstrates the effectiveness of IONet in solving PBE with a real small-molecule ADP, further research is needed to investigate the computational efficiency of PI-IONet and other neural network-based methods in solving large-scale computational problems in biophysics, such as solving PBE with real macromolecules. We will postpone this part of the work to future research.

In this example, we approximate the solution operator of PBE \eqref{eq: NPBE} using PI-IONet, DD-IONet and DD-DeepONet. Herein, the weights in the loss function \eqref{eq: empirical loss function} of PI-IONet are set as $\lambda_1=\lambda_2=\lambda_3=1$ and $\lambda_4=100$, and the network hyperparameters are presented in Table \ref{tab: 3d pb}. In all cases, the neural networks are trained after $5 \times 10^4$ parameter updates. The test errors and training costs of these three models are shown in Table \ref{tab: 3d pb}.
It can be observed that DD-IONet achieves the lowest relative $L^2$ error. Additionally, despite lacking any paired input-output measurements, except for the boundary conditions, the error accuracy of PI-IONet is comparable to those of DD-DeepONet, albeit with slightly higher training costs. In addition, a visual comparison of the reference and predicted surface potentials of the protein ADP is shown in Fig. \ref{fig: adp plot}. These findings further emphasize the capability of IONet to effectively handle parametric interface problems within irregular interface. While our current work demonstrates the effectiveness of IONet in solving PBE with a real small-molecule ADP, further research is needed to investigate the computational efficiency of PI-IONet and other neural network-based methods in solving large-scale computational problems in biophysics, such as solving PBE with real macromolecules. We will postpone this part of the work to future research.

\subsection{Parametric elliptic interface problems in six dimensions\label{Ex: example5 6d}}
\hypertarget{example6}{\textbf{Example} 6}. Our final example aims to highlight the ability of the proposed framework to handle high-dimensional parametric interface problems. Here, we consider Eq. \eqref{eq: interface pde} defined on a 6-dimension sphere of radius 0.6 domain $\Omega$ enclosing another smaller 6-dimension sphere of radius 0.5 as the interior domain $\Omega_1$.  Our goal is to learn the solution operator mapping from the source term $f$ to the latent solution of Eq. \eqref{eq: interface pde}, i.e., $\mathcal{G}: f(\mathbf{x})\rightarrow u(\mathbf{x})$.  In this case, $f$ has the following specific forms
\begin{equation*}
    f(\mathbf{x}) =\left\{
\begin{aligned}
&-p_1\prod_{i=1}^6{\exp(x_i)},\ &&\mathbf{x} \in \Omega_1, \\
&-p_2\prod_{i=1}^6{\sin(x_i)},\ &&\mathbf{x}\in \Omega_2,
\end{aligned}\right.
\end{equation*}
where $(p_1,p_2)$ randomly sample from $[1,10]\times[-10^{-2},-10^{-3}]$. For the problem setup, the coefficient
\begin{equation*}
    a(\mathbf{x}) =\left\{
\begin{aligned}
&1,\ &&\mathbf{x} \in \Omega_1, \\
&10^{-3},\ &&\mathbf{x}\in \Omega_2,
\end{aligned}\right.
\end{equation*}
has a large contrast ($a_1/a_2=10^3$), the boundary condition is given as
$h(\mathbf{x})=\prod_{i=1}^6{\sin(x_i)}$, the interface conditions are chosen as
$g_D(\mathbf{x}) = \prod_{i=1}^6{\sin(x_i)}-\prod_{i=1}^6{\exp(x_i)},
$ and
\begin{equation*}
    g_N(\mathbf{x}) = \frac{5}{3}\left(10^{-3}\sum_{i=1}^6\left(x_i\cos(x_i){\prod_{j=1, j\neq i}^6{\sin(x_j)}}\right)-\left(\sum_{i=1}^6{x_i}\right)\prod_{j=1}^6{\exp(x_j)}\right).
\end{equation*}
Note that, when  we take $f$ with $p_1=6$ and $p_2=-6\times10^{-3}$, Eq. \eqref{eq: interface pde} exists an exact solution \cite{hu2022discontinuity}
\begin{equation*}
    u(\mathbf{x}) =\left\{
\begin{aligned}
&\prod_{i=1}^6{\exp(x_i)},\ &&\mathbf{x} \in \Omega_1, \\
&\prod_{i=1}^6{\sin(x_i)},\ &&\mathbf{x}\in \Omega_2.
\end{aligned}\right.
\end{equation*}

\begin{table}[htbp]
\centering
\fontsize{10}{8}\selectfont
\begin{threeparttable}
\caption{$L^\infty$ and relative $L^2$ errors for PI-IONet with different widths and depths. These errors are obtained by averaging the results from three independent experiments, each involving different network random initialization and randomly generated training dataset.}
\begin{tabular}{cccccc}
\toprule
Depth& Width & \#Parameters &$\norm{u-u_{\theta}}_{\infty}$ & $\frac{\norm{u-u_{\theta}}_{2}}{\norm{u}_{2}}$ & Training time (hours)\\
 \midrule
3& 30 & 9k
& 3.35e-2$\pm$1.95e-2&
1.38e-2$\pm$6.18e-4 & 0.42\\
4& 40  & 21k
&  1.36e-2$\pm$6.06e-3&
8.04e-3$\pm$4.01e-3 & 0.55\\
5& 50  & 43k
&7.67e-3$\pm$1.46e-3&
1.56e-3$\pm$6.84e-4 & 0.66\\
\bottomrule
\label{tab: 6d case}
\end{tabular}
\end{threeparttable}
\end{table}

In this study, we train PI-IONet $\mathcal{G}_{\mathbf{\theta}}$ with different scales of architecture to approximate the solution operator and then test the trained model in the case of input $f$ with $p_1=6$ and $p_2=-6\times10^{-3}$. In all cases, we train PI-IONet by minimizing the loss function \eqref{eq: empirical loss function} with $\lambda_1=\lambda_2=\lambda_3=1$ and $\lambda_4=100$ for $4\times10^4$ iterations of parameter updates, utilizing 100 randomly sampled input functions. The sensors for discretizing the input functions comprise a set of randomly sampled points in the computational domain, as opposed to lattice points in a higher dimensional space. Specifically, we set the number of sensors as 40.

Table \ref{tab: 6d case} illustrates the test errors measured at $10,000$ test data points over the whole domain and the computational costs for training PI-IONet. It is shown that the accuracy of the numerical solution improves while the training cost grows as the number of trainable parameters increases. The final $L^\infty$ error and relative $L^2$ error are about $0.7\%$ and $0.2\%$, respectively, with low deviations. These findings suggest that PI-IONet has the potential to achieve high performance in dealing with high-dimensional output solutions of parametric elliptic interface problems, even in scenarios involving high-contrast coefficients.

\section{Conclusions\label{conclusions}}
In this work, we have investigated deep neural network-based operator learning methods and proposed the interfaced operator network (IONet) to tackle parametric elliptic interface problems. The main contribution is that we first combine the domain-decomposed method with the operator learning methods and employ multiple branch nets and trunk nets to explicitly handle the discontinuities across the interface in the input and output functions. In addition, we introduce tailored physics-informed loss designed to constrain the physical consistency of the proposed model. This strategy reduces the requirement for training data and empowers the IONet to remain effective even in the absence of paired input-output training data. We also provide theory and numerical experiments to demonstrate that the proposed IONet is effective and reliable for approximating the solution operator of parametric interface problems. In our simulations, we systematically studied the effects of different factors on the accuracy of IONet and existing state-of-the-art operator networks. The results show that the proposed IONet is more robust and accurate in dealing with many kinds of parametric interface problems due to its discontinuity-preserving architecture.

Despite the preliminary success, there are still many issues that need further investigation. As an advantage of neural operators is their fast predictions, an important aspect of interest regarding IONet is the systematic comparison of its computational cost with other numerical methods for solving interface problems. In addition, one limitation is the absence of treating geometry configuration as an input function in our current work. Integrating the geometry configuration into IONet could potentially enhance its capabilities and further broaden its range of applications. Furthermore, we have not yet obtained the convergence rate for IONet, which would provide valuable insights into both the accuracy and stability of the model. Inspired by the works on the error estimates for DeepONets \cite{lanthaler2022error} and generalization performance analysis of deep learning for PDEs \cite{lu2022machine, shin2020on,luo2020two,jiao2021convergence}, including interface problems \cite{wu2023Convergence}, it is interesting to improve the convergence properties and the error estimation of IONet for solving parametric elliptic interface problems. Moreover, IONet can be viewed as a specific DeepONet preserving discontinuity, and we are also interested in exploring the feasibility of integrating recent advancements and extensions from DeepONet  (e.g., DeepONet with proper orthogonal decomposition \cite{lu2022comprehensive}, DeepONet based on latent representations and autoencoders \cite{kontolati2023learning}, DeepONet using Laplace transform \cite{cao2023lno} ) into IONet.

\section*{Acknowledgments}
This research was funded by the Strategic Priority Research Program of Chinese Academy of Sciences (Grant No. XDB0500000) and the National Natural Science Foundation of China (Grant Nos. 12371413, 22073110 and 12171466).

\appendix
\section{Proof of Theorem \ref{thm:approximation of ino}}\label{app:proof of approximation}
Here we provide the proof of Theorem \ref{thm:approximation of ino}, which relies on the universal approximation theorem of FNNs (see, e.g., \cite{de2021approximation}) and the tensor product decomposition of operators \cite{jin2022mionet}.
\begin{proof}
Denote $T_{\ind} = \{a|_{\overline{\Omega}_{\ind}} |\ a\in T\}\subset C(\overline{\Omega}_{\ind})$ where ${\ind}=1,\cdots, \Ind$. Consider an operator $\mathcal{\tilde{G}}$ mapping from $ T_1\times \cdots \times T_{\Ind}$ to $X(\Omega)$ as
\begin{equation}\label{es:constructG}
\mathcal{\tilde{G}}(a_1,\cdots,  a_{\Ind}) := \mathcal{G}(a),\quad \text{where}\ a(x) = a_i(x) ,\quad \text{if }x\in \Omega_i.
\end{equation}
Observing the fact that the function values evaluated at given points are equivalent to the piece-wise linear functions of Faber-Schauder basis, and by corollary 2.6 in \cite{jin2022mionet}, for any $\varepsilon>0$, there exist positive integers $m_{\ind}$, $K$ and continuous functions $g_{i,k} \in C(\mathbb{R}^{m_i})$ and $u_k \in X(\Omega)$ and $\y_{\ind}^1,\cdots, \y^{\ind}_{m_{\ind}} \in \Omega_{\ind}$, where $k=1, \cdots, K$, $\ind=1, \cdots, \Ind$, such that
\begin{equation}\label{es:tildeG}
\sup_{a_1\in T_1, \ a_2\in T_2} \norm{\mathcal{\tilde{G}}(a_1, \cdots, a_{\Ind})(\cdot) - \sum_{k=1}^K \prod_{\ind=1}^{\Ind} g_{\ind,k}(a_{\ind}(\y_{\ind}^1),\cdots, a_{\ind}(\y^{\ind}_{m_{\ind}})) \cdot u_k(\cdot)}_{X(\Omega)} \leq \frac{\varepsilon}{2}.
\end{equation}
Since $T_{\ind}$ is a collection of continuous functions, and $\Omega$ is bounded, we have that $A_{\ind} := \{ (a_{\ind}(\y_{\ind}^1),\cdots, a_{\ind}(\y_{\ind}^{m_{\ind}})) | a_{\ind} \in T_{\ind}\} \in \mathbb{R}^{m_{\ind}}$ is bounded for ${\ind}=1,\cdots, \Ind$, and there exists a cuboid containing $A_{\ind}$. Consequently, by the approximation theorems of tanh FNN \cite{de2021approximation}, for any $\delta>0$, there exist tanh FNNs $\mathcal{N}_{b_{\ind}}: \mathbb{R}^{m_{\ind}}\rightarrow \mathbb{R}^K$, $\mathcal{N}^{{\ind}}_{t}: \Omega_i\rightarrow \mathbb{R}^K$, such that
\begin{equation*}
\norm{[\mathcal{N}_{b_{\ind}}]_k - g_{{\ind},k}}_{C^{\infty}(A_{\ind})} \leq \delta,
\quad
\norm{[\mathcal{N}_{t}^{\ind}]_k - u_k|_{\Omega_{\ind}}}_{H^2(\Omega_{\ind})}\leq \delta,\quad {\ind}=1,\cdots, \Ind,\quad k=1,\cdots,K,
\end{equation*}
where $[\mathcal{N}]_k$ denotes the $k$-th component of $\mathcal{N}$. Denote
\begin{equation*}
M_k = \max \{\norm{g_{{1},k}}_{C^{\infty}(A_{1})},\cdots,\ \norm{g_{{\Ind},k}}_{C^{\infty}(A_{\Ind})}, \norm{u_k}_{X(\Omega)}\},\  k=1,\cdots,K.
\end{equation*}
Subsequently, we can choose a sufficiently small $\delta$ such that
\begin{equation}\label{es:nn}
\begin{aligned}
&\Big\| \sum_{k=1}^K \prod_{\ind=1}^{\Ind}[\mathcal{N}_{b_{\ind}}]_k(a_{\ind}(\y_{\ind}^1),\cdots, a_{\ind}(\y^{\ind}_{m_{\ind}})) \cdot [\mathcal{N}_{t}^{\ind}]_k(\cdot) - \sum_{k=1}^K\prod_{\ind=1}^{\Ind} g_{\ind,k}(a_{\ind}(\y_{\ind}^1),\cdots, a_{\ind}(\y^{\ind}_{m_{\ind}})) \cdot u_k(\cdot) \Big\|_{H^2(\Omega_i)}\\
\leq& \sum_{k=1}^K (M_k+\delta)^{I+1} - M_k^{I+1}\leq \frac{\varepsilon}{2}.
\end{aligned}
\end{equation}
Finally, combining definition (\ref{es:constructG}), estimates (\ref{es:tildeG}) and (\ref{es:nn}), we conclude that for $\ind=1,\cdots, \Ind$,
\begin{equation*}
\begin{aligned}
&\sup_{a\in T} \norm{\mathcal{G}(a)(\cdot) - \mathcal{S} \left( \mathcal{N}_{b_1}(a(\y_1^1),\cdots, a(\y^1_{m_1})) \odot \mathcal{N}_{b_2}(a(\y^2_1),\cdots, a(\y^2_{m_2})) \odot \mathcal{N}^{\ind}_{t}(\cdot) \right) }_{H^2(\Omega_{\ind})} \\
\leq& \sup_{a \in T} \norm{\mathcal{\tilde{G}}(a|_{\Omega_1}, \cdots, a|_{\Omega_{\Ind}})(\cdot) - \sum_{k=1}^K\prod_{\ind=1}^{\Ind} g_{\ind,k}(a_{\ind}(\y_{\ind}^1),\cdots, a_{\ind}(\y^{\ind}_{m_{\ind}})) \cdot u_k(\cdot) }_{H^2(\Omega_{\ind})}\\
& + \sup_{a\in T}\norm{\sum_{k=1}^K \prod_{\ind=1}^{\Ind}[\mathcal{N}_{b_{\ind}}]_k(a_{\ind}(\y_{\ind}^1),\cdots, a_{\ind}(\y^{\ind}_{m_{\ind}})) \cdot [\mathcal{N}_{t}^{\ind}]_k(\cdot) - \sum_{k=1}^K\prod_{\ind=1}^{\Ind} g_{\ind,k}(a_{\ind}(\y_{\ind}^1),\cdots, a_{\ind}(\y^{\ind}_{m_{\ind}})) \cdot u_k(\cdot) }_{H^2(\Omega_i)}\\
\leq & \frac{\varepsilon}{2} + \frac{\varepsilon}{2} = \varepsilon,
\end{aligned}
\end{equation*}
which completes the proof.
\end{proof}

\bibliographystyle{elsarticle-num}
\bibliography{ref}
\end{document}